\documentclass[12pt]{amsart}

\usepackage{amsmath,amssymb,amsthm}
\usepackage[all]{xy}
\usepackage[right]{lineno}
\usepackage{mathrsfs}
\usepackage[dvipdfm,colorlinks=true]{hyperref}
\usepackage{epic,eepic}
\usepackage{bm}

\numberwithin{equation}{section}

\SelectTips{eu}{12}

\newtheorem{theorem}{Theorem}[section]
\newtheorem{proposition}[theorem]{Proposition}
\newtheorem{lemma}[theorem]{Lemma}
\newtheorem{corollary}[theorem]{Corollary}
\newtheorem{problem}[theorem]{Problem}

\theoremstyle{definition}
\newtheorem{definition}[theorem]{Definition}
\newtheorem{example}[theorem]{Example}
\newtheorem{conjecture}[theorem]{Conjecture}

\theoremstyle{remark}
\newtheorem{remark}[theorem]{Remark}

\newcommand{\R}{\mathbb{R}}
\newcommand{\RZ}{\R\mathcal{Z}}
\newcommand{\Z}{\mathcal{Z}}
\newcommand{\W}{\mathcal{W}}
\newcommand{\st}{\mathrm{st}}
\newcommand{\lk}{\mathrm{lk}}
\newcommand{\dl}{\mathrm{dl}}
\newcommand{\Sd}{\mathrm{Sd}}
\newcommand{\Cone}{\mathrm{Cone}}

\title[Fat wedge filtration and decomposition of polyhedral products]{Fat wedge filtration and decomposition of polyhedral products}

\author{Kouyemon Iriye}
\address{Department of Mathematics and Information Sciences, Osaka Prefecture University, Sakai, 599-8531, Japan}
\email{kiriye@mi.s.osakafu-u.ac.jp}
\author{Daisuke Kishimoto}
\address{Department of Mathematics, Kyoto University, Kyoto, 606-8502, Japan}
\email{kishi@math.kyoto-u.ac.jp}
\thanks{K.I. is supported by JSPS KAKENHI (No. 26400094), and D.K. is supported by JSPS KAKENHI (No. 25400087)}


\begin{document}

\maketitle

\begin{abstract}
The polyhedral product constructed from a collection of pairs of cones and their bases and a simplicial complex $K$ is studied by investigating its filtration called the fat wedge filtration. We give a sufficient condition for decomposing the polyhedral product in terms of the fat wedge filtration of the real moment-angle complex for $K$, which is a desuspension of the decomposition of the suspension of the polyhedral product due to Bahri, Bendersky, Cohen, and Gitler \cite{BBCG}. We show that the condition also implies a strong connection with the Golodness of $K$, and is satisfied when $K$ is dual sequentially Cohen-Macaulay over $\mathbb{Z}$ or $\lceil\frac{\dim K}{2}\rceil$-neighborly so that the polyhedral product decomposes. Specializing to the moment-angle complex, we prove that the similar condition on its fat wedge filtrations is necessary and sufficient for its decomposition. 
\end{abstract}

\baselineskip 17pt


\section{Introduction}

Let $K$ be an abstract simplicial complex on the vertex set $[m]:=\{1,\ldots,m\}$, and let $(\underline{X},\underline{A})$ be a collection of pairs of spaces indexed by the vertices of $K$. The space $\Z_K(\underline{X},\underline{A})$ which is now called the {\it polyhedral product} is defined by the union of product spaces constructed from $(\underline{X},\underline{A})$ in accordance with the combinatorial information of $K$. Polyhedral products were first found in Porter's work on higher order Whitehead products \cite{P} in 1965, and appear in several fundamental constructions in algebra, geometry, and topology related with combinatorics: the cohomology of $\Z_K(\mathbb{C}P^\infty,*)$ and $\Z_K(D^2,S^1)$ are identified with the Stanley-Reisner ring of $K$ and its certain derived algebra, respectively \cite{DJ,BBP,BP}; the fundamental group of $\Z_K(\R P^\infty,*)$ and $\Z_K(D^1,S^0)$ are the right-angled Coxeter group of the 1-skeleton of $K$ and its commutator subgroup \cite{DO}; the union of the coordinate subspace arrangement in $\R^m$ associated with $K$ is $\Z_K(\R,*)$, and its complement has the homotopy type of $\Z_K(D^1,S^0)$ \cite{GT1,IK1,BP}. From these examples, one sees that the special polyhedral products $\Z_K(C\underline{X},\underline{X})$ and $\Z_K(\underline{X},\underline{*})$ are especially important, where $(C\underline{X},\underline{X})$ and $(\underline{X},\underline{*})$ are collections of pairs of cones and their base spaces, and spaces and their basepoints, respectively. There is a homotopy fibration involving these polyhedral products, so they are supplementary to each other in a sense. The object to study in this paper is the polyhedral product $\Z_K(C\underline{X},\underline{X})$, and we are particularly interested in its homotopy type. 

Among other results on the homotopy types of polyhedral products, the work of Bahri, Bendersky, Cohen, and Gitler \cite{BBCG} is remarkable. They proved a decomposition of a suspension of $\Z_K(\underline{X},\underline{A})$ in general, and specializing to the polyhedral product $\Z_K(C\underline{X},\underline{X})$, they obtained the following decomposition, where the notations will be explained later.

\begin{theorem}
[Bahri, Bendersky, Cohen, and Gitler \cite{BBCG}]
\label{BBCG}
There is a homotopy equivalence
$$\Sigma\Z_K(C\underline{X},\underline{X})\simeq\Sigma\bigvee_{\emptyset\ne I\subset[m]}|\Sigma K_I|\wedge\widehat{X}^I.$$
\end{theorem} 

Let us call the decomposition of this theorem the BBCG decomposition. The proof of the BBCG decomposition is a combination of the decomposition of suspensions of general polyhedral products which they obtained, and a formula of homotopy colimits \cite{ZZ}. Unfortunately, from the original proof, one cannot seize the intrinsic nature of $\Z_K(C\underline{X},\underline{X})$ which yields the BBCG decomposition, but the BBCG decomposition certainly showed a direction in studying the homotopy type of $\Z_K(C\underline{X},\underline{X})$, that is, to describe the homotopy type of $\Z_K(C\underline{X},\underline{X})$ by desuspending the BBCG decomposition. This direction was proposed in \cite{BBCG} when $K$ is a special simplicial complex called a shifted complex: they conjectured that the previous result of Grbi\'c and Theriault \cite{GT1} on $\Z_K(D^2,S^1)$ when $K$ is a shifted complex, can be generalized to a desuspension of the BBCG decomposition. This conjecture was affirmatively resolved by the authors \cite{IK1}, and was partially generalized to dual vertex-decomposable complexes by Welker and Gruji\'c \cite{GW}, where Grbi\'c and Theriault \cite{GT2} also considered a desuspension for shifted complexes but the paper includes serious mistakes such as the closedness of $\mathcal{W}_n$ by retracts in the proof of the main theorem. However, the proofs of these results are quite ad-hoc so that we cannot apply them to other classes of simplicial complexes.

The first aim of this paper is to elucidate the intrinsic nature of the polyhedral product $\Z_K(C\underline{X},\underline{X})$ for general $K$ which yields the BBCG decomposition and its desuspension. The structure of $\Z_K(C\underline{X},\underline{X})$ in question will be described by a certain filtration which we call the {\it fat wedge filtration}. We will see that the BBCG decomposition is actually a consequence of the property of the fat wedge filtration such that it splits after a suspension, so the analysis of the fat wedge filtration naturally shows a way to desuspend the BBCG decomposition. In analyzing the fat wedge filtration, the special polyhedral product $\Z_K(D^1,S^0)$ which is called the {\it real moment-angle complex} for $K$ and is denoted by $\RZ_K$ plays the fundamental role, where the real moment-angle complexes have been studied in toric topology as a rich source producing manifolds with good 2-torus actions. We will prove that the fat wedge filtration of $\RZ_K$ is a cone decomposition of $\RZ_K$, and will describe the attaching maps of the cones explicitly in a combinatorial manner. We say that the fat wedge filtration of $\RZ_K$ is trivial if all of these attaching maps are null homotopic, and now state our first main result.

\begin{theorem}
\label{main-decomp}
If the fat wedge filtration of $\RZ_K$ is trivial, then for any $\underline{X}$ there is a homotopy equivalence
$$\Z_K(C\underline{X},\underline{X})\simeq\bigvee_{\emptyset\ne I\subset[m]}|\Sigma K_I|\wedge\widehat{X}^I.$$
\end{theorem}

As well as $\RZ_K$ the polyhedral product $\Z_K(D^2,S^1)$ has been studied in toric topology as an object producing manifolds with good torus actions, which is called the {\it moment-angle complex} for $K$ and is denoted by $\Z_K$. We will prove that the fat wedge filtration of $\Z_K$ is also a cone decomposition, so we can define its triviality as well as that of $\RZ_K$. We will give two conditions equivalent to the triviality of the fat wedge filtration of $\Z_K$ as follows.

\begin{theorem}
\label{main-decomp-Z}
The following three conditions are equivalent:
\begin{enumerate}
\item The fat wedge filtration of $\Z_K$ is trivial;
\item $\Z_K$ is a co-H-space;
\item There is a homotopy equivalence
$$\Z_K\simeq\bigvee_{\emptyset\ne I\subset[m]}\Sigma^{|I|+1}|K_I|.$$
\end{enumerate}
\end{theorem} 


Note that if the BBCG decomposition desuspends, then $\Z_K(C\underline{X},\underline{X})$ becomes a suspension by Theorem \ref{main-decomp}, so in particular, all products and (higher) Massey products in the cohomology of $\Z_K$ are trivial. As mentioned above, the cohomology of $\Z_K$ is isomorphic to a certain derived algebra of the Stanley-Reisner ring of $K$, and the triviality of products and (higher) Massey products of this derived algebra is equivalent to a property of $K$ called the {\it Golodness} which has been extensively studied in combinatorial commutative algebra \cite{HRW,BJ,B}. We will also show the triviality of the fat wedge filtration of $\RZ_K$ (resp. $\Z_K$) implies the (resp. stable) homotopy version of the Golodness of $K$. 

The second aim of this paper is to examine the triviality of the fat wedge filtration of the real moment-angle complexes for specific simplicial complexes which implies the decomposition of polyhedral products by Theorem \ref{main-decomp}. To this end, we must choose appropriate classes of simplicial complexes. For shifted and dual vertex-decomposable complexes, desuspensions of the BBCG decomposition were studied in \cite{GT1,IK1,GW} as mentioned above, where dual shifted complexes are shifted. Originally, shifted and vertex-decomposable complexes were introduced as handy subclasses of shellable complexes in \cite{BW}, and shellable complexes are combinatorial simplification of sequentially Cohen-Macaulay (SCM, for short) complexes over $\mathbb{Z}$ \cite{S,BWW}, where sequentially Cohen-Macaulayness is a non-pure generalization of Cohen-Macaulayness. Then there are implications:
\begin{equation}
\label{SCM-implication}
\text{shifted}\quad\Longrightarrow\quad\text{vertex-decomposable}\quad\Longrightarrow\quad\text{shellable}\quad\Longrightarrow\quad\text{SCM over }\mathbb{Z}
\end{equation}
Then we first choose dual shellable complexes to show the triviality of the fat wedge filtrations of real moment-angle complexes, and then generalize its argument homologically to obtain the following result for dual SCM complexes over $\mathbb{Z}$, which is a substantial improvement of the previous results \cite{GT1,IK1,GW}. The theorem will be actually proved for a larger class of simplicial complexes including dual SCM complexes over $\mathbb{Z}$, and a spin off of the method used for dual shellable complexes will be given to produce a $p$-local desuspension of the BBCG decomposition for another class of simplicial complexes including dual sequentially Cohen-Macaulay complexes over $\mathbb{Z}/p$ under a mild condition on $\underline{X}$.

\begin{theorem}
\label{main-SCM}
If $K$ is dual SCM over $\mathbb{Z}$, then the fat wedge filtration of $\RZ_K$ is trivial.
\end{theorem}

We will prove that $|\Sigma K_I|$ has the homotopy type of a wedge of spheres for any $\emptyset\ne I\subset[m]$ if $K$ is dual SCM over $\mathbb{Z}$, so we obtain the following by Theorem \ref{main-decomp} and \ref{main-SCM}.

\begin{corollary}
\label{main-SCM-sphere}
If $K$ is dual SCM over $\mathbb{Z}$, then $\Z_K(D^n,S^{n-1})$ is homotopy equivalent to a wedge of spheres.
\end{corollary}

We next consider the property of the inductive triviality of the attaching maps of the cones in the fat wedge filtration of $\RZ_K$. When all attaching maps in the $i^\text{th}$ filter are trivial, we will show that the attaching maps for the $(i+1)^\text{th}$ filter become trivial after composed with a certain map, say $\alpha$. So the attaching maps lift to the homotopy fiber of $\alpha$. By evaluating the connectivity of the homotopy fiber of $\alpha$, we will obtain the following, where neighborly complexes are defined in Definition \ref{def-neighborly}. The theorem will be slightly generalized by replacing dimension with homology dimension. 

\begin{theorem}
\label{main-neighborly}
If $K$ is $\lceil\frac{\dim K}{2}\rceil$-neighborly, then the fat wedge filtration of $\RZ_K$ is trivial.
\end{theorem}

This paper is organized as follows. In Section 2 we define polyhedral products, and collect some of their examples and properties which will be used later. In Section 3 we combinatorially describe the fat wedge filtration of $\RZ_K$, and in Section 4, we study the fat wedge filtration of $\Z_K(C\underline{X},\underline{X})$ by using the description of the fat wedge filtration of $\RZ_K$. We then prove Theorem \ref{main-decomp}. In Section 5 we further investigate the fat wedge filtration of $\Z_K$, and prove Theorem \ref{main-decomp-Z}. Section 6 deals with a connection between the triviality of the fat wedge filtrations of $\RZ_K$ and $\Z_K$ and the Golodness of $K$. In Section 7 and 8, we give criteria, called the fillability and the homology fillability, for the triviality of the fat wedge filtration of $\RZ_K$, and apply them to dual shellable complexes and dual sequentially Cohen-Macaulay complexes over $\mathbb{Z}$, proving Theorem \ref{main-SCM}. Section 9 is a spin off of the arguments for dual shellable complexes in Section 8. We introduce a new simplicial complexes called extractible complexes, and prove a $p$-local desuspension of the BBCG decompsition for them under a mild condition on $\underline{X}$ for extractible complexes over $\mathbb{Z}/p$. In Section 10 we give another criterion for the triviality of the fat wedge filtration of $\RZ_K$ which is Theorem \ref{main-neighborly}. Finally in Section 11, we give a list of possible future problems on the fat wedge filtration of polyhedral products.

Throughout the paper, we use the following notations:
\begin{itemize}
\item Let $K$ be a simplicial complex on the vertex set $[m]$, where we put $[m]:=\{1,\ldots,m\}$;
\item Let $|K|$ denote the geometric realization of $K$;
\item Let $\underline{X}$ be a sequence of spaces with non-degenerate basepoints $\{X_i\}_{i\in[m]}$;
\item Put $(C\underline{X},\underline{X}):=\{(CX_i,X_i)\}_{i\in[m]}$, pairs of reduced cones and their base spaces;
\item If $(X,A)$ is a pair of spaces, the symbol $(X,A)$ also denotes its $m$-copies ambiguously.
\end{itemize}

The authors are grateful to Piotr Beben for discussion on the homotopy Golodness. Thanks also goes to Tatsuya Yano for careful reading of the draft and kindness to allow us to show his example in Section 11, and also to Lukas Katth\"an for letting us know his and De Stafani's results.


\section{Definition of polyhedral products}

In this section, we define polyhedral products, and show some examples.  Then we recall a homotopy fibration involving polyhedral products that we will use.

\begin{definition}
Let $(\underline{X},\underline{A})$ be a sequence of pairs of spaces $\{(X_i,A_i)\}_{i\in[m]}$. The polyhedral product $\Z_K(\underline{X},\underline{A})$ is defined by
$$\Z_K(\underline{X},\underline{A}):=\bigcup_{\sigma\in K}(\underline{X},\underline{A})^\sigma\quad(\subset X_1\times\cdots\times X_m)$$
where $(\underline{X},\underline{A})^\sigma=Y_1\times\cdots\times Y_m$ for $Y_i=X_i$ for $i\in\sigma$ and $A_i$ for $i\not\in\sigma$.
\end{definition}

The special polyhedral product $\Z_K(D^1,S^0)$ and $\Z_K(D^2,S^1)$ are called the real moment-angle complex for $K$ and the moment-angle complex for $K$ and are denoted by $\Z_K$ and $\RZ_K$, respectively, as mentioned above. We here give three easy examples of polyhedral products.

\begin{example}
\label{wedge-example}
If $K$ is the simplicial complex with discrete $m$-points, then we have
$$\Z_K(\underline{X},\underline{*})=X_1\vee\cdots\vee X_m.$$
On the other hand, if $K$ is the boundary of the full $(m-1)$-simplex, then $\Z_K(\underline{X},\underline{*})$ is the fat wedge of $X_1,\ldots,X_m$. More generally, if $K$ is the $(k-1)$-skeleton of the full $(m-1)$-simplex, then $\Z_K(\underline{X},\underline{*})$ is the $k^\text{th}$ generalized fat wedge 
$$T^k:=\{(x_1,\ldots,x_m)\in X_1\times\cdots\times X_m\,\vert\,\text{at least }m-k\text{ of }x_i\text{ are basepoints}\}$$
of $X_1,\ldots,X_m$.
\end{example}

\begin{example}
\label{boundary-example}
When $m=2$ and $K$ is the boundary of the full 1-simplex, we have
$$\Z_K(C\underline{X},\underline{X})=(CX_1\times X_2)\cup(X_1\times CX_2)=X_1*X_2$$
where $X*Y$ means the join of $X$ and $Y$. For general $m$, if $K$ is the boundary of the full $(m-1)$-simplex, it is proved in \cite{P} that there is a homotopy equivalence
$$\Z_K(C\underline{X},\underline{X})\simeq\Sigma^{m-1}X_1\wedge\cdots\wedge X_m$$
which can be recovered by the results of \cite{IK1}, or more generally by Theorem \ref{main-decomp}. If $K$ is a skeleton of the full $(m-1)$-simplex, the homotopy type of $\Z_K(C\underline{X},\underline{X})$ can be described also by \cite{IK1} or Theorem \ref{main-decomp}.
\end{example}

\begin{example}
\label{non-desuspension-example}
We observe the polyhedral product of the joint of two simplicial complexes. We set notation. For simplicial complexes $K_1,K_2$ on disjoint vertex sets, their join is defined by
$$K_1*K_2:=\{\sigma_1\sqcup\sigma_2\,\vert\,\sigma_1\in K_1,\,\sigma_2\in K_2\}.$$
Let $I$ be a non-empty subset of $[m]$, and let $K_I$ denote the full subcomplex of $K$ on $I$, that is, $K_I:=\{\sigma\subset I\,\vert\,\sigma\in K\}$. For a sequence of pairs of spaces $(\underline{X},\underline{A})=\{(X_i,A_i)\}_{i\in[m]}$, we put $(\underline{X}_I,\underline{A}_I):=\{(X_i,A_i)\}_{i\in I}$. We can deduce the following immediately from the definition of polyhedral products. For $\emptyset\ne I,J\subset[m]$ with $I\cap J=\emptyset$ and $I\cup J=[m]$, we have
$$\Z_{K_I*K_J}(\underline{X},\underline{A})\cong\Z_{K_I}(\underline{X}_I,\underline{A}_I)\times\Z_{K_J}(\underline{X}_J,\underline{A}_J).$$
Then we see that the polyhedral product $\Z_K(C\underline{X},\underline{X})$ is not always a suspension, which in particular implies that the BBCG decomposition does not always desuspend: for example, if $m=4$ and $K$ is a square which is the join of 2-copies of the simplicial complex with discrete 2-points, we have $\RZ_K\cong S^1\times S^1$ by Example \ref{boundary-example} and the above homeomorphism.
\end{example}

We recall from \cite{DS} a homotopy fibration involving polyhedral products, and we here produce an alternative proof.

\begin{lemma}
[cf. {\cite[Proposition, pp.180]{Fa}}]
\label{hocolim-fibration}
Let $\{F_i\to E_i\to B\}_{i\in I}$ be a diagram of homotopy fibrations over a fixed base $B$. Then 
$$\underset{i\in I}{\mathrm{hocolim}}\,F_i\to \underset{i\in I}{\mathrm{hocolim}}\,E_i\to B$$
is a homotopy fibration.
\end{lemma}

\begin{proposition}
[Denham and Suciu \cite{DS}]
\label{fibration}
There is a homotopy fibration
$$\Z_K(C\Omega\underline{X},\Omega\underline{X})\to\Z_K(\underline{X},*)\xrightarrow{\rm incl}X_1\times\cdots\times X_m.$$
\end{proposition}

\begin{proof}
For any $\sigma\subset[m]$ there is a homotopy fibration $(C\Omega\underline{X},\Omega\underline{X})^\sigma\to(\underline{X},*)^\sigma\xrightarrow{\rm incl}X_1\times\cdots\times X_m$ which is natural with respect to the inclusions of subsets of $[m]$. Then we have a diagram of homotopy fibrations $\{(C\Omega\underline{X},\Omega\underline{X})^\sigma\to(\underline{X},*)^\sigma\xrightarrow{\rm incl}X_1\times\cdots\times X_m\}_{\sigma\in K}$, so it follows from Lemma \ref{hocolim-fibration} that there is a homotopy fibration
$$\underset{\sigma\in K}{\mathrm{hocolim}}\,(C\Omega\underline{X},\Omega\underline{X})^\sigma\to\underset{\sigma\in K}{\mathrm{hocolim}}\,(\underline{X},*)^\sigma\to X_1\times\cdots\times X_m.$$
Since the maps $(C\Omega\underline{X},\Omega\underline{X})^\sigma\to(C\Omega\underline{X},\Omega\underline{X})^\tau$ and $(\underline{X},*)^\sigma\to(\underline{X},*)^\tau$ are cofibrations for all $\sigma\subset\tau\subset[m]$, the above homotopy colimits are naturally homotopy equivalent to the colimits which are $\Z_K(C\Omega\underline{X},\Omega\underline{X})$ and $\Z_K(\underline{X},*)$, completing the proof.
\end{proof}


\section{Fat wedge filtration of $\RZ_K$}

In this section, we introduce the fat wedge filtration of $\Z_K(C\underline{X},\underline{X})$ and investigate the filtration of the real moment-angle complex $\RZ_K$. We first define the fat wedge filtration of a general subspace of a product of spaces. Let $T^k$ be the $k^\text{th}$ generalized fat wedge of $X_1,\ldots,X_m$ as in Example \ref{wedge-example} for $k=0,\ldots,m$. Then we get a filtration
$$*=T^0\subset T^1\subset\cdots\subset T^m=X_1\times\cdots\times X_m.$$
For a subspace $Y\subset X_1\times\cdots\times X_m$ including the base point of $X_1\times\cdots\times X_m$, we put $Y^k:=Y\cap T^k$ for $i=0,\ldots,m$, so we get a filtration
$$*=Y^0\subset Y^1\subset\cdots\subset Y^m=Y$$
which is called the fat wedge filtration of $Y$.

We give a combinatorial description of the fat wedge filtration of the real moment-angle complex $\RZ_K$, where we choose the point $-1$ to be the basepoint of $S^0=\{-1,+1\}$. For any $\emptyset\ne I\subset[m]$ we identify $\RZ_{K_I}$ with the subspace $\{(x_1,\ldots,x_m)\in\RZ_K\,\vert\,x_i=-1\text{ for }i\not\in I\}$ of $\RZ_K$. Then by the definition of the fat wedge filtration, we have
\begin{equation}
\label{RZ^i}
\RZ_K^0=\{(-1,\ldots,-1)\}\quad\text{and}\quad\RZ_K^i=\bigcup_{I\subset[m],\,|I|=i}\RZ_{K_I}
\end{equation}
for $i=1,\ldots,m$. In order to describe the fat wedge filtration of $\RZ_K$ combinatorially, we employ the cubical decomposition of a simplicial complex presented in \cite{BP}. To nested subsets $\sigma\subset\tau\subset[m]$ (not necessarily simplices of $K$), we assign the $(|\tau|-|\sigma|)$-dimensional face
$$C_{\sigma\subset\tau}:=\{(x_1,\ldots,x_m)\in(D^1)^{\times m}\,\vert\,x_i=-1,+1\text{ according as }i\in\sigma\text{ and }i\not\in\tau\}$$
of the cube $(D^1)^{\times m}$. Notice that any face of the cube $(D^1)^{\times m}$ is expressed by $C_{\sigma\subset\tau}$ for some $\sigma\subset\tau\subset[m]$, and in particular, any vertex of $(D^1)^{\times m}$ is given by $C_{\sigma\subset\sigma}$ for some $\sigma\subset[m]$. Let $\Sd L$ denote the barycentric subdivision of a simplicial complex $L$. Then the vertices of $\Sd\Delta^{[m]}$ are non-empty subsets of $[m]$, so we can define a piecewise linear map
$$i_c\colon|\Sd\Delta^{[m]}|\to(D^1)^{\times m},\quad\sigma\mapsto C_{\sigma\subset\sigma}$$
which is an embedding onto the union of $(m-1)$-dimensional faces of $(D^1)^{\times m}$ including the vertex $(-1,\ldots,-1)$, where $\Delta^{[m]}$ denotes the simplex on the vertex set $[m]$. This embedding is the cubical decomposition of $\Delta^{[m]}$, where one can see the reason for the name ``cubical decomposition'' from Figure 1.

\begin{figure}[htbp]
\begin{center}
\setlength\unitlength{1mm} 
\begin{picture}(100,33)(0,-2)
\Thicklines
\drawline(20,0)(0,30)(40,30)(20,0)
\drawline(10,15)(20,20)(20,30)
\drawline(20,20)(30,15)
\put(20,20){\whiten\circle{2.2}}
\put(20,30){\circle*{2.2}}
\put(10,15){\circle*{2.2}}
\put(30,15){\circle*{2.2}}
\put(0,30){\shade\circle{2.2}}
\put(40,30){\shade\circle{2.2}}
\put(20,0){\shade\circle{2.2}}
\drawline(70,27)(70,7)(80,-3)(80,17)(70,27)
\drawline(80,-3)(100,2)(100,22)(80,17)
\drawline(70,7)(79,9.3)
\drawline(81,9.7)(90,12)(100,2)
\thinlines
\dashline{1.5}(70,27)(90,32)(100,22)
\dashline{1.5}(90,32)(90,12)
\Thicklines
\put(80,-3){\whiten\circle{2.2}}
\put(70,7){\circle*{2.2}}
\put(100,2){\circle*{2.2}}
\put(80,17){\circle*{2.2}}
\put(70,27){\shade\circle{2.2}}
\put(90,12){\shade\circle{2.2}}
\put(100,22){\shade\circle{2.2}}
\thicklines
\put(43,15){\vector(1,0){20}}
\put(52,18){$i_c$}
\end{picture}
\caption{the embedding $i_c$ for $m=3$}
\end{center}
\end{figure}
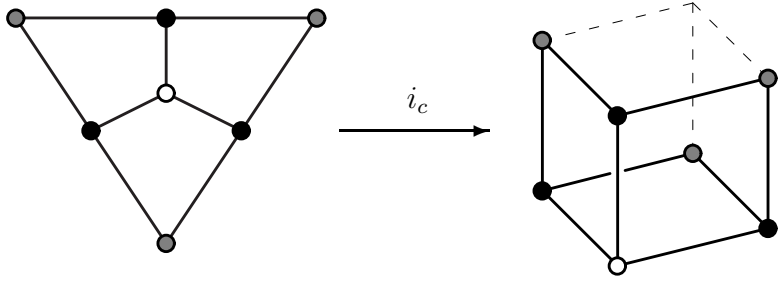

We define the cone and the suspension of $K$ by
$$\mathrm{Cone}(K):=\Delta^{[1]}*K\quad\text{and}\quad\Sigma K:=\partial\Delta^{[2]}*K$$
as usual. By extending the embedding $i_c$, we get a piecewise linear homeomorphism
$$\mathrm{Cone}(i_c)\colon|\mathrm{Cone}(\Sd\Delta^{[m]})|\to(D^1)^{\times m}$$
which sends the cone point of $|\mathrm{Cone}(\Sd\Delta^{[m]})|$ to the vertex $(+1,\ldots,+1)\in(D^1)^{\times m}$. Since the vertex set of $K$ is $[m]$, $K$ is a subcomplex of $\Delta^{[m]}$. Then by restricting $i_c$ and $\mathrm{Cone}(i_c)$, we obtain the embeddings
$$i_c\colon|\Sd K|\to(D^1)^{\times m},\quad\mathrm{Cone}(i_c)\colon|\mathrm{Cone}(\Sd K)|\to(D^1)^{\times m}$$
which are the cubical decompositions of $K$ and $\mathrm{Cone}(K)$. 

We express the difference $\mathrm{Cone}(i_c)(|\mathrm{Cone}(\Sd K)|)-i_c(|\Sd K|)$ in terms of the faces $C_{\sigma\subset\tau}$. For any $\tau\subset[m]$ we have $i_c(|\Sd\tau|)=\bigcup_{\emptyset\ne\sigma\subset\tau}C_{\sigma\subset\tau}$ and $\mathrm{Cone}(i_c)(|\mathrm{Cone}(\Sd\tau)|)=\bigcup_{\sigma\subset\tau}C_{\sigma\subset\tau}$, so we get
\begin{equation}
\label{i_c-image}
i_c(|\Sd K|)=\bigcup_{\emptyset\ne\sigma\subset\tau\in K}C_{\sigma\subset\tau}\quad\text{and}\quad\Cone(i_c)(|\Cone(\Sd K)|)=\bigcup_{\sigma\subset\tau\in K}C_{\sigma\subset\tau}.
\end{equation}
Then it follows that
\begin{equation}
\label{difference}
\Cone(i_c)(|\Cone(\Sd K)|)-i_c(|\Sd K|)=\bigcup_{\tau\in K}C_{\emptyset\subset\tau}-\bigcup_{\emptyset\ne\sigma\subset\tau\in K}C_{\sigma\subset\tau}.
\end{equation}
We next express $\RZ_K^i$ in terms of the faces $C_{\sigma\subset\tau}$ as well, and show that the cubical decompositions of full subcomplexes of $K$ naturally come into the fat wedge filtration of $\RZ_K$. We denote by $(D^1_I,S^0_I)$ the $|I|$-copies of the pair $(D^1,S^0)$ for $I\subset[m]$. Then for $\mu\subset I$, we have $(D^1_I,S^0_I)^\mu=\bigcup_{\sigma\subset\tau\subset I,\,\tau-\sigma=\mu}C_{\sigma\subset\tau}^I$, where $C_{\sigma\subset\tau}^I$ is the face $C_{\sigma\subset\tau}$ of $(D^1)^{\times I}$. We get
$$\RZ_{K_I}=\bigcup_{\mu\in K,\,\mu\subset I}(D^1_I,S^0_I)^\mu=\bigcup_{\mu\in K,\,\mu\subset I}\left(\bigcup_{\sigma\subset\tau\subset I,\,\tau-\sigma=\mu}C_{\sigma\subset\tau}^I\right)=\bigcup_{\sigma\subset\tau\subset I,\,\tau-\sigma\in K}C_{\sigma\subset\tau}^I$$
and 
\begin{equation}
\label{RZ-C}
\begin{aligned}
\RZ_{K_I}^{|I|-1}&=\bigcup_{\substack{J\subset I\\|J|=|I|-1}}\left(\bigcup_{\mu\in K,\,\mu\subset J}(D^1_I,S^0_I)^\mu\right)&&=\bigcup_{\substack{J\subset I\\|J|=|I|-1}}\left(\bigcup_{\mu\in K,\,\mu\subset J}\left(\bigcup_{\sigma\subset\tau\subset J,\,\tau-\sigma=\mu}C_{\sigma\subset\tau}^J\right)\right)\\
&=\bigcup_{\mu\in K,\,\mu\subset I}\left(\bigcup_{\emptyset\ne\sigma\subset\tau\subset I,\,\tau-\sigma=\mu}C_{\sigma\subset\tau}^I\right)&&=\bigcup_{\emptyset\ne\sigma\subset\tau\subset I,\,\tau-\sigma\in K}C_{\sigma\subset\tau}^I.
\end{aligned}
\end{equation}
Then by \eqref{i_c-image}, the embedding $\mathrm{Cone}(i_c)\colon|\mathrm{Cone}(\Sd K_I)|\to(D^1)^{\times I}$ descends to a map 
\begin{equation}
\label{Cone(i_c)}
(|\mathrm{Cone}(\Sd K_I)|,|\Sd K_I|)\to(\RZ_{K_I},\RZ_{K_I}^{|I|-1}).
\end{equation}
Moreover since
$$\RZ_{K_I}-\RZ_{K_I}^{|I|-1}=\bigcup_{\emptyset\subset\tau\subset I,\,\tau\in K}C_{\emptyset\subset\tau}^I-\bigcup_{\emptyset\ne\sigma\subset\tau\in K}C_{\sigma\subset\tau}^I=\Cone(i_c)(|\Cone(\Sd K_I)|)-i_c(|\Sd K_I|)$$
by \eqref{difference}, the map \eqref{Cone(i_c)} is actually a relative homeomorphism. Then as
$$\RZ_K^i-\RZ_K^{i-1}=\coprod_{I\subset[m],\,|I|=i}(\RZ_{K_I}-\RZ_{K_I}^{|I|-1}),$$
the disjoint union of the maps \eqref{Cone(i_c)}
$$\coprod_{I\subset[m],\,|I|=i}(|\mathrm{Cone}(\Sd K_I)|,|\Sd K_I|)\to(\RZ_K^i,\RZ_K^{i-1})$$
turns out to be a relative homeomorphism. Let $\varphi_{K_I}$ denote the map $|\Sd K_I|\to\RZ_{K_I}^{|I|-1}$ in \eqref{Cone(i_c)}. Then we have established:

\begin{theorem}
\label{cone-decomp}
For $i=1,\ldots,m$, $\RZ_K^i$ is obtained from $\RZ_K^{i-1}$ by attaching cones by $j_{K_I}\circ\varphi_{K_I}$ for all $I\subset[m]$ with $|I|=i$, where $j_{K_I}\colon\RZ_{K_I}^{i-1}\to\RZ_K^{i-1}$ is the inclusion.
\end{theorem}

The above theorem shows that the fat wedge filtration of $\RZ_K$ is a cone decomposition in the usual sense. We say that the fat wedge filtration of $\RZ_K$ is trivial if the maps $\varphi_{K_I}$ are null homotopic for all $\emptyset\ne I\subset[m]$. Since $\RZ_{K_I}^{|I|-1}$ is a retract of $\RZ_K^{|I|-1}$, this is equivalent to the composite $j_{K_I}\circ\varphi_{K_I}$ is null homotopic for any $\emptyset\ne I\subset[m]$. We here consider two cases in which the fat wedge filtration of $\RZ_K$ is trivial. We first consider the flag complex of a chordal graph as in \cite{GPTW}. Here graphs mean one dimensional simplicial complexes, and the flag complex of a graph $\Gamma$ is the simplicial complex whose $n$-simplices are complete graphs with $n+1$ vertices in $\Gamma$. Recall that a graph is called chordal if its minimal cycles are of length at most 3. 

\begin{proposition}
\label{flag-varphi}
If $K$ is the flag complex of a chordal graph, then the fat wedge filtration of $\RZ_K$ is trivial.
\end{proposition}

\begin{proof}
Suppose $K$ is the flag complex of a graph $\Gamma$. It is known that $\Gamma$ is chordal if and only if each component of $K$ is contractible. Then since $\RZ_K^{m-1}$ is path-connected, $\varphi_K$ is null homotopic. For any $\emptyset\ne I\subset[m]$, the full subgraph $\Gamma_I$ is chordal, and $K_I$ is the flag complex of $\Gamma_I$. Then we obtain that $\varphi_{K_I}$ is null homotopic for any $\emptyset\ne I\subset[m]$.
\end{proof}

We next consider the case $\dim K\ge m-2$. We start with observing properties of the map $\varphi_K$ for general $K$. In \cite{IK2}, it is proved that the inclusion $\Sigma\RZ_{K_I}^{|I|-1}\to\Sigma\RZ_{K_I}$ admits a left homotopy inverse for $\emptyset\ne I\subset[m]$ which is obtained by patching together the retraction $\RZ_{K_I}\to\RZ_{K_J}$ for $J\subset I$. Then by Theorem \ref{cone-decomp}, we have the following.

\begin{proposition}
\label{suspension-varphi}
The maps $\Sigma\varphi_{K_I}$ are null homotopic for all $\emptyset\ne I\subset[m]$.
\end{proposition}

For a simplex $\sigma$ of $K$, we denote the deletion and the link of $\sigma$ by $\dl_K(\sigma)$ and $\lk_K(\sigma)$, that is, $\dl_K(\sigma)=K_{[m]-\sigma}$ and $\lk_K(\sigma)=\{\tau\subset[m]\,\vert\,\sigma\cap\tau=\emptyset,\,\sigma\cup\tau\in K\}$. Since 
$$\RZ^{m-1}_K=(\RZ_{\dl_K(v)}^{m-2}\times S^0)\cup(\RZ_{\dl_K(v)}\times\{-1\})\cup(\RZ_{\lk_K(v)}^{m-2}\times D^1)$$ 
for a vertex $v$ of $K$, there is a projection $\RZ_K^{m-1}\to\Sigma\RZ_{\lk_K(v)}^{m-2}$ which pinches $\RZ_{\dl_K(v)}^{m-2}\times\{+1\}$ and $\RZ_{\dl_K(v)}\times\{-1\}$ to two points. Then it is straightforward to check that through the identification $|\Sd K|/|\Sd(\dl_K(v))|=\Sigma|\Sd(\lk_K(v))|$, we have a commutative diagram
$$\xymatrix{|\Sd K|\ar[rr]^{\varphi_K}\ar[d]^{\rm proj}&&\RZ_K^{m-1}\ar[d]^{\rm proj}\\
\Sigma|\Sd(\lk_K(v))|\ar[rr]^{\Sigma\varphi_{\lk_K(v)}}&&\Sigma\RZ_{\lk_K(v)}^{m-2}.}$$
So by Proposition \ref{suspension-varphi}, we get:

\begin{corollary}
\label{trivial-proj}
The composite $|\Sd K|\xrightarrow{\varphi_K}\RZ_K^{m-1}\xrightarrow{\rm proj}\Sigma\RZ_{\lk_K(v)}^{m-2}$ is null homotopic.
\end{corollary}

\begin{proposition}
\label{dim=m-2}
If $\dim K\ge m-2$, then the fat wedge filtration of $\RZ_K$ is trivial.
\end{proposition}

\begin{proof}
As $\dim K\ge m-2$, there is a vertex $v$ of $K$ such that $\dl_K(v)$ is the full simplex $\Delta^{[m]-v}$, implying $\RZ_{\dl_K(v)}$ is contractible. So the projection $\RZ_K^{m-1}\to\Sigma\RZ_{\lk_K(v)}^{m-2}$ is a homotopy equivalence, hence $\varphi_K$ is null homotopic by Corollary \ref{trivial-proj}. Since $\dim K_I\ge |I|-2$ for all $\emptyset\ne I\subset[m]$, the map $\varphi_{K_I}$ is null homotopic for each $\emptyset\ne I\subset[m]$ by the same observation.
\end{proof}


\section{Fat wedge filtration of $\Z_K(C\underline{X},\underline{X})$}

In this section, we investigate the fat wedge filtration of $\Z_K(C\underline{X},\underline{X})$ by using the maps $\varphi_{K_I}$ obtained in the previous section, and we prove Theorem \ref{main-decomp}.

As well as the real moment-angle complexes, we may regard $\Z_{K_I}(C\underline{X}_I,\underline{X}_I)$ for $\emptyset\ne I\subset[m]$ as a subspace of $\Z_K(C\underline{X},\underline{X})$ which is in fact a retract since every $X_i$ has a basepoint, so we have
$$\Z_K^0(C\underline{X},\underline{X})=*\quad\text{and}\quad\Z_K^i(C\underline{X},\underline{X})=\bigcup_{I\subset[m],\,|I|=i}\Z_{K_I}(C\underline{X}_I,\underline{X}_I)$$
for $i=1,\ldots,m$. We describe $\Z_{K_I}(C\underline{X}_I,\underline{X}_I)$ by using the map $\varphi_{K_I}$. Let $I=\{j_1<\cdots<j_i\}$ be a subset of $[m]$ and put $\underline{X}^{\times I}=X_{j_1}\times\cdots\times X_{j_i}$. Consider the composite of maps
\begin{equation}
\label{i_c-RZ-Z}
|\mathrm{Cone}(\Sd K_I)|\times\underline{X}^{\times I}\xrightarrow{\mathrm{Cone}(i_c)\times 1}\RZ_{K_I}\times\underline{X}^{\times I}\to CX_{j_1}\times\cdots\times CX_{j_i}
\end{equation}
where the second arrow maps $((t_1,\ldots,t_i);(x_1,\ldots,x_i))$ to $((t_1,x_1);.\ldots;(t_i,x_i))$ for $t_k\in D^1,x_k\in X_{j_k}$. One easily deduces that the composite descends to a surjection
$$\Phi_{K_I}\colon|\mathrm{Cone}(\Sd K_I)|\times\underline{X}^{\times I}\to\Z_{K_I}(C\underline{X}_I,\underline{X}_I)$$
which is homeomorphic on $|\mathrm{Cone}(\Sd K_I)|\times\underline{X}^{\times I}-\Phi_{K_I}^{-1}(\Z_{K_I}^{i-1}(C\underline{X}_I,\underline{X}_I))$, and since we are using reduced cones, we have
$$\Phi_{K_I}^{-1}(\Z_{K_I}^{i-1}(C\underline{X}_I,\underline{X}_I))=(|\mathrm{Cone}(\Sd K_I)|\times T^{i-1}(\underline{X}_I))\cup(|\Sd K_I|\times\underline{X}^{\times I}).$$
So we obtain a relative homeomorphism
$$\Phi_{K_I}\colon(|\mathrm{Cone}(\Sd K_I)|,|\Sd K_I|)\times(\underline{X}^{\times I},T^{i-1}(\underline{X}_I))\to(\Z_{K_I}(C\underline{X}_I,\underline{X}_I),\Z_{K_I}^{i-1}(C\underline{X}_I,\underline{X}_I))$$
where a product of pairs of spaces are given by $(A,B)\times(C,D)=(A\times B,(A\times D)\cup(B\times C))$ as usual. Then since
$$\Z_K^i(C\underline{X},\underline{X})-\Z_K^{i-1}(C\underline{X},\underline{X})=\coprod_{I\subset[m],\,|I|=i}(\Z_{K_I}(C\underline{X}_I,\underline{X}_I)-\Z_{K_I}^{i-1}(C\underline{X}_I,\underline{X}_I)),$$
we obtain the following.

\begin{theorem}
\label{Phi}
The map
$$\coprod_{I\subset[m],\,|I|=i}\Phi_{K_I}\colon\coprod_{I\subset[m],\,|I|=i}(|\mathrm{Cone}(\Sd K_I)|,|\Sd K_I|)\times(\underline{X}^{\times I},T^{i-1}(\underline{X}_I))\to(\Z_K^i(C\underline{X},\underline{X}),\Z_K^{i-1}(C\underline{X},\underline{X}))$$
is a relative homeomorphism.
\end{theorem}

Recall that a categorical sequence of a space $Y$ in the sense of Fox \cite{Fo} is a filtration $*=Y_0\subset Y_1\subset\cdots\subset Y_m=Y$ such that the inclusion $Y_i-Y_{i-1}\to Y$ is null homotopic for $i=1,\ldots,m$. By the above theorem one can easily deduce that the fat wedge filtration of $\Z_K(C\underline{X},\underline{X})$ is a categorical sequence whereas the fat wedge filtration of $\RZ_K$ is a cone decomposition, where a cone decomposition is a special categorical sequence. In the special case, we can show that this categorical sequence is a cone decomposition as in Section 5, but the authors do not know whether this is true or not in general.

One can reprove Theorem \ref{BBCG} by using Theorem \ref{Phi}, from which one can interpret more directly how full subcomplexes of $K$ appear in the BBCG decomposition.

\begin{corollary}
[Bahri, Bendersky, Cohen, and Gitler \cite{BBCG}]
\label{BBCG-natural}
There is a homotopy equivalence
$$\Sigma\Z_K(C\underline{X},\underline{X})\xrightarrow{\simeq}\Sigma\bigvee_{\emptyset\ne I\subset[m]}|\Sigma K_I|\wedge\widehat{X}^I$$
which is natural with respect to $\underline{X}$ and inclusions of subcomplexes of $K$, where $\widehat{X}^I=\bigwedge_{i\in I}X_i$.
\end{corollary}

\begin{proof}
Note that for $\emptyset\ne I\subset[m]$, $\Z_{K_I}(C\underline{X}_I,\underline{X}_I)$ is a retract of $\Z_K(C\underline{X},\underline{X})$ such that for $\emptyset\ne J\subset[m]$ there is a commutative diagram
$$\xymatrix{\Z_{K_I}(C\underline{X}_I,\underline{X}_I)\ar[r]^{\rm incl}\ar[d]^{\rm proj}&\Z_{K_{I\cup J}}(C\underline{X}_{I\cup J},\underline{X}_{I\cup J})\ar[d]^{\rm proj}\\
\Z_{K_{I\cap J}}(C\underline{X}_{I\cap J},\underline{X}_{I\cap J})\ar[r]^{\rm incl}&\Z_{K_J}(C\underline{X}_J,\underline{X}_J).}$$
Then $\{\Z_{K_I}(C\underline{X}_I,\underline{X}_I)\}_{I\subset[m]}$ is a space over the poset $2^{[m]}$ with natural retractions in the sense of \cite{IK2}, where $2^{[m]}$ is the power set of $[m]$ with the inclusion ordering and we put $\Z_{K_\emptyset}(C\underline{X}_\emptyset,\underline{X}_\emptyset)$ to be a point. Then the theorem follows from \cite{IK2}.
\end{proof}

\begin{remark}
The BBCG decomposition is obtained also by the retractile argument of James \cite{Ja}, but it is hard to get the naturality by the retractile argument.
\end{remark}

From the description of the fat wedge filtration of $\Z_K(C\underline{X},\underline{X})$ in \eqref{i_c-RZ-Z} and Theorem \ref{Phi}, one sees that the attaching maps $\varphi_{K_I}$ of the cone decomposition of $\RZ_K$ control the fat wedge filtration of $\Z_K(C\underline{X},\underline{X})$. We further investigate this control in the extreme case that the fat wedge filtration of $\RZ_K$ is trivial, that is, we prove Theorem \ref{main-decomp}. We will use the following technical lemma.

\begin{lemma}
\label{relative-homeo}
Let $(X,A),(Y,B)$ be NDR pairs. Suppose that $Y$ has the quotient topology by a relative homeomorphism $f\colon(X,A)\to(Y,B)$, and that the restriction $f\vert_A$ is null homotopic in $B$. Then there is a string of homotopy equivalences
$$Y\xleftarrow{\simeq}D_f\xrightarrow{\simeq}B\vee Y/B$$
which is natural with respect to the relative homeomorphism between NDR pairs satisfying the same conditions and having compatible homotopies.
\end{lemma}

\begin{proof}
Let $D_f$ be the double mapping cylinder of $f\vert_A$ and the inclusion $A\to X$. Since $(X, A)$ and $(Y, B)$ are NDR pairs and $f\vert_A$ is null homotopic in $B$, there is a string of homotopy equivalences $B\vee X/A\xleftarrow{\simeq}D_f\xrightarrow{\simeq}B\cup_fX$. Since $Y$ is given the quotient topology by $f$, we have $X/A\cong Y/B$ and $B\cup_fX\cong Y$. Then we obtain the desired string. The naturality of the string follows from the naturality of double mapping cylinders.
\end{proof}

We first consider the special case that the basepoint of $X_i$ is isolated for each $i$.

\begin{lemma}
\label{Z-split-isolated}
If the fat wedge filtration of $\RZ_K$ is trivial and the basepoint $*_i$ of $X_i$ is isolated for each $i$, then there is a string of homotopy equivalences
$$\Z_K(C\underline{X},\underline{X})\xleftarrow{\delta}D(\underline{X})\xrightarrow{\epsilon}\bigvee_{\emptyset\ne I\subset[m]}|\Sigma K_I|\wedge\widehat{X}^I$$
which is natural with respect to $\underline{X}$.
\end{lemma}

\begin{proof}
The lemma follows from Theorem \ref{Phi} and Lemma \ref{relative-homeo} if we show that the restriction 
$$(|\Cone(\Sd K_I)|\times T^{|I|-1}(\underline{X}_I))\cup(|\Sd K_I|\times X^I)\to\Z_{K_I}^{|I|-1}(C\underline{X}_I,\underline{X}_I)$$
of $\Phi_{K_I}$ is null homotopic for all $\emptyset\ne I\subset[m]$ by a homotopy which is natural with respect to $\underline{X}$. Since all basepoints are isolated, we have
$$(|\Cone(\Sd K_I)|\times T^{|I|-1}(\underline{X}_I))\cup(|\Sd K_I|\times X^I)=(|\Cone(\Sd K_I)|\times T^{|I|-1}(\underline{X}_I))\sqcup(|\Sd K_I|\times\check{X}^I)$$
where $\check{X}^I=(X_{j_1}-*_{j_1})\times\cdots\times(X_{j_i}-*_{j_i})$ and $I=\{j_1<\cdots<j_i\}$. Then we can consider the restriction of $\Phi_{K_I}$ to $|\Cone(\Sd K_I)|\times T^{|I|-1}(\underline{X}_I)$ and $|\Sd K_I|\times\check{X}^I$ independently. By deforming $|\Cone(\Sd K_I)|$ to its cone point, the restriction of $\Phi_{K_I}$ to $|\Cone(\Sd K_I)|\times T^{|I|-1}(\underline{X}_I)$ is naturally homotopic to the inclusion $T^{|I|-1}(\underline{X}_I)\to\Z_{K_I}^{|I|-1}(C\underline{X}_I,\underline{X}_I)$ which is also naturally null homotopic. Note that the restriction of $\Phi_{K_I}$ to $|\Sd K_I|\times\check{X}^I$ factors so that
$$|\Sd K_I|\times\check{X}^I\xrightarrow{\varphi_{K_I}\times 1}\RZ_{K_I}^{|I|-1}\times\check{X}^I\to\Z_{K_I}^{|I|-1}(C\underline{X}_I,\underline{X}_I)$$
by the construction of $\Phi_{K_I}$. Then by assumption it is naturally homotopic to the inclusion $\check{X}^I\to\Z_{K_I}^{|I|-1}(C\underline{X}_I,\underline{X}_I)$ which is also naturally null homotopic. Therefore the proof is completed.
\end{proof}

\begin{proof}[Proof of Theorem \ref{main-decomp}]
For a space $A$, we always assume that the basepoint of $A\sqcup *$ is $*$ even when $A$ itself has a basepoint. We define $\underline{X}^k=\{X_i^k\}_{i\in[m]}$ by
$$X_i^k:=\begin{cases}X_i&i\le k\\X_i\sqcup*&i>k\end{cases}$$
for $k=0,\ldots,m$, where we may allow $X_i=\emptyset$ for $i>k$. We also define $\underline{X}^{(k)}=\{X_i^{(k)}\}_{i\in[m]}$ and $\underline{X}^{[k]}=\{X_i^{[k]}\}_{i\in[m]}$ by
$$X_i^{(k)}:=\begin{cases}X_i^k&i\ne k+1\\\ast_{k+1}\sqcup*&i=k+1\end{cases}\quad\text{and}\quad X_i^{[k]}:=\begin{cases}X_i^k&i\ne k+1\\\ast_{k+1}&i=k+1\end{cases}$$
for $k=0,\ldots,m$, where $*_i$ is the basepoint of $X_i$ as above. Note that $\underline{X}^{(k)}$ and $\underline{X}^{[k]}$ are the special cases of $\underline{X}^k$ when $X_{k+1}$ is $*_{k+1}$ and $\emptyset$, respectively. Note also that a map $f:\underline{X}\to\underline{Y}$ induces maps $f^k:\underline{X}^k\to\underline{Y}^k$, $f^{(k)}:\underline{X}^{(k)}\to\underline{Y}^{(k)}$ and $f^{[k]}:\underline{X}^{[k]}\to\underline{Y}^{[k]}$. Let $\iota:\underline{X}^{(k)}\to\underline{X}^k$ and $\pi:\underline{X}^{(k)}\to\underline{X}^{[k]}$ denote the inclusion and the projection, respectively. Then $\iota$ and $\pi$ are natural with respect to $\underline{X}$, that is, $f^k\circ\iota=\iota\circ f^{(k)}$ and $f^{[k]}\circ\pi=\pi\circ f^{(k)}$.  

Put $\W_K(\underline{X})=\bigvee_{\emptyset\ne I\subset[m]}|\Sigma K_I|\wedge\widehat{X}^I$. We construct a string of homotopy equivalences $\Z_K(C\underline{X}^k,\underline{X}^k)\xleftarrow{\delta^k}D(\underline{X}^k)\xrightarrow{\epsilon^k}\W_K(\underline{X}^k)$ which is natural with respect to $\underline{X}$  by induction on $k$, so we obtain the desired homotopy equivalence when $k=m$ since $\underline{X}^m=\underline{X}$. We define the string of homotopy equivalences for $k=0$ by Lemma \ref{Z-split-isolated} which is natural in $\underline{X}$. Suppose that we have constructed a string of homotopy equivalences $\Z_K(C\underline{X}^k,\underline{X}^k)\xleftarrow{\delta^k}D(\underline{X}^k)\xrightarrow{\epsilon^k}\W_K(\underline{X}^k)$ which is natural in $\underline{X}$. Then there is a commutative diagram
$$\xymatrix{\Z_K(C\underline{X}^{[k]},\underline{X}^{[k]})&\Z_K(C\underline{X}^{(k)},\underline{X}^{(k)})\ar[r]^\iota\ar[l]_\pi&\Z_K(C\underline{X}^k,\underline{X}^k)\\
D(\underline{X}^{[k]})\ar[d]^{\epsilon^k}\ar[u]_{\delta^k}&D(\underline{X}^{(k)})\ar[d]^{\epsilon^k}\ar[r]^\iota\ar[l]_\pi\ar[u]_{\delta^k}&D(\underline{X}^k)\ar[d]^{\epsilon^k}\ar[u]_{\delta^k}\\
\W_K(\underline{X}^{[k]})&\W_K(\underline{X}^{(k)})\ar[r]^\iota\ar[l]_\pi&\W_K(\underline{X}^k)}$$
which is natural with respect to $\underline{X}$. Observe that the pushouts of the top and the bottom rows are $\Z_K(C\underline{X}^{k+1},\underline{X}^{k+1})$ and $\W_K(\underline{X}^{k+1})$, respectively. We put $D(\underline{X}^{k+1})$  to be the double mapping cylinder of the middle row and a string of maps $\Z_K(C\underline{X}^{k+1},\underline{X}^{k+1})\xleftarrow{\delta^{k+1}}D(\underline{X}^{k+1})\xrightarrow{\epsilon^{k+1}}\W_K(\underline{X}^{k+1})$ by the induced maps. Then since the maps $\iota$ in the top and the bottom rows are cofibrations, it follows from the standard argument on double mapping cylinders that $\epsilon^{k+1}$ and $\delta^{k+1}$ are homotopy equivalences. Moreover, since the diagram is natural with respect to $\underline{X}$, so is the new string also. Therefore the induction proceeds.
\end{proof}


\section{Co-H-structure on $\Z_K$}

In this section we further investigate the fat wedge filtration of the moment-angle complex $\Z_K$, and prove the equivalence between its triviality and a co-H-structure on $\Z_K$, which is Theorem \ref{main-decomp-Z}.

By Theorem \ref{Phi} there is a relative homeomorphism
$$(|\mathrm{Cone}(\Sd K)|,|\Sd K|)\times((S^1)^m,T^{m-1}(S^1))\to(\Z_K,\Z_K^{m-1}).$$
Let $\omega\colon S^{m-1}\to T^{m-1}(S^1)$ be the higher Whitehead product \cite{P} which is the attaching map of the top cell of $(S^1)^m$. Then $\omega$ induces a relative homeomorphism
$$(D^m,S^{m-1})\to((S^1)^m,T^{m-1}(S^1)),$$
so we get a relative homeomorphism
$$(|\mathrm{Cone}(\Sd K)|,|\Sd K|)\times(D^m,S^{m-1})\to(\Z_K,\Z_K^{m-1}).$$
By definition the left hand side is 
$$(|\mathrm{Cone}(\Sd K)|\times D^m,(|\mathrm{Cone}(\Sd K)|\times S^{m-1})\cup(|\Sd K|\times D^m))\\=(|\mathrm{Cone}(\Sd K)|\times D^m,|\Sd K|* S^{m-1}).$$
Then since $\mathrm{Cone}(\Sd K)|\times D^m$ is contractible, we obtain that $\Z_K$ is obtained from $\Z_K^{m-1}$ by attaching a cone to some map $\widehat{\varphi}_K\colon|\Sd K|* S^{m-1}\to\Z_K^{m-1}$, where there is actually a relative homeomorphism between $(C(X*Y),X*Y)$ and $(CX\times CY,X*Y)$ for spaces $X,Y$. Thus we obtain:

\begin{theorem}
\label{cone-decomp-Z}
For $i=1,\ldots,m$, $\Z_K^i$ is obtained from $\Z_K^{i-1}$ by attaching a cone to the composite $|\Sd K_I|*S^{i-1}\xrightarrow{\widehat{\varphi}_{K_I}}\Z_{K_I}^{i-1}\xrightarrow{\rm incl}\Z_K^{i-1}$ for each $I\subset[m]$ with $|I|=i$.
\end{theorem}

\begin{remark}
Since higher Whitehead products are defined for suspension spaces, Theorem \ref{cone-decomp-Z} is true for $\Z_K(C\underline{X},\underline{X})$ when each $X_i$ is a suspension.
\end{remark}

To prove Theorem \ref{main-decomp-Z} we consider the connectivity of $\Z_K(C\underline{X},\underline{X})$.

\begin{proposition}
\label{Z-1-conn}
If each $X_i$ is path-connected, then $\Z_K(C\underline{X},\underline{X})$ is simply connected.
\end{proposition}

\begin{proof}
As in \cite{GT1,IK1,GW}, for a vertex $v$ of $K$, the pushout of simplicial complexes
$$\xymatrix{\lk_K(v)\ar[r]\ar[d]&\mathrm{st}_K(v)\ar[d]\\
\dl_K(v)\ar[r]&K}$$
induces a pushout of spaces
\begin{equation}
\label{Z-pushout}
\xymatrix{\Z_{\lk_K(v)}(C\underline{X}_{[m]-v},\underline{X}_{[m]-v})\times X_v\ar[r]\ar[d]&\Z_{\lk_K(v)}(C\underline{X}_{[m]-v},\underline{X}_{[m]-v})\times CX_v\ar[d]\\
\Z_{\dl_K(v)}(C\underline{X}_{[m]-v},\underline{X}_{[m]-v})\times X_v\ar[r]&\Z_K(C\underline{X},\underline{X})}
\end{equation}
where $\st_K(v)$ denotes the star of $v$ in $K$, i.e. $\st_K(v):=\lk_K(v)*\{v\}$, and all arrows in \eqref{Z-pushout} are cofibrations. The proof is completed by an inductive application of the van Kampen theorem to the pushout \eqref{Z-pushout} which is also a homotopy pushout. 
\end{proof}

\begin{proof}[Proof of Theorem \ref{main-decomp-Z}]
By Theorem \ref{cone-decomp-Z} (1) implies (3), and (3) obviously implies (2). So we prove (2) implies (1).   Induct on $m$. For $m=1$ there is nothing to do. Assume that (2) implies (1) for $K$ with vertices less than $m$. We now suppose $\Z_K$ is a co-H-space. Then $\Z_{K_I}$ is also a co-H-space for any $\emptyset\ne I\subsetneq[m]$ since it is a retract of $\Z_K$. Hence by the induction hypothesis we get $\widehat{\varphi}_{K_I}\simeq*$ for any $\emptyset\ne I\subsetneq[m]$, so it remains to show that $\widehat{\varphi}_K$ is null homotopic. By Theorem \ref{cone-decomp-Z} it is sufficient to prove that the inclusion $j\colon\Z_K^{m-1}\to\Z_K$ has a left homotopy inverse. Since $\widehat{\varphi}_{K_I}\simeq*$ for any $\emptyset\ne I\subsetneq[m]$, $\Z_K^{m-1}$ is a suspension, so there is a diagram
$$\xymatrix{\Z_K^{m-1}\ar[d]^j\ar[r]&\bigvee^{2^m-2}\Z_K^{m-1}\ar[d]^{\bigvee^{2^m-2}j}\ar[r]^q&\bigvee_{\emptyset\ne I\subsetneq[m]}\Z_{K_I}\ar@{=}[d]\\
\Z_K\ar[r]&\bigvee^{2^m-2}\Z_K\ar[r]^{q'}&\bigvee_{\emptyset\ne I\subsetneq[m]}\Z_{K_I}}$$
where the left horizontal arrows are defined by co-multiplications and $q,q'$ are the projections. Then this diagram commutes after a suspension. The proof of Theorem \ref{BBCG} (cf. \cite{IK2}) implies that the composite  of the upper horizontal arrows is an isomorphism in homology, so a homotopy equivalence by Proposition \ref{Z-1-conn} and the J.H.C. Whitehead theorem. Then we obtain that $j$ admits a left homotopy inverse, so the attaching map $\varphi_K$ is null homotopic. Therefore the proof is completed.
\end{proof}


\section{Golodness and fat wedge filtrations}

In this section, we study a relation between the Golodness of a simplicial complex $K$ and the triviality of the fat wedge filtrations of $\RZ_K$ and $\Z_K$. We first recall the definition of the Golodness of simplicial complexes. Let $\Bbbk$ be a commutative ring. Recall that the Stanley-Reisner ring of a simplicial complex $K$ over $\Bbbk$ is defined by
$$\Bbbk[K]:=\Bbbk[v_1,\ldots,v_m]/\mathcal{I}_K,\quad|v_i|=2$$
where $\mathcal{I}_K$ is the ideal generated by monomials $v_{i_1}\cdots v_{i_k}$ for $\{i_1,\ldots,i_k\}\not\in K$. As is well known, Stanley-Reisner rings have been a constant source of interest in algebra and combinatorics, and have been producing a variety of results and applications. See \cite{S} for general structures of Stanley-Reisner rings. The Golodness of $K$ over $\Bbbk$ is formally defined by the equality between certain formal power series involving the Poincar\'e series of $\Bbbk[K]$ and its cohomology \cite{G}, where there is always a non-strict inequality between these formal power series. We here give an equivalent condition to the Golodness here as the definition of the Golodness, where the equivalence was proved by Golod \cite{G}. We consider one of the most important derived algebras of the Stanley-Reisner ring $\Bbbk[K]$
$$\mathrm{Tor}_{\Bbbk[v_1,\ldots,v_m]}^*(\Bbbk[K],\Bbbk)$$
where the product structure is induced from the Koszul resolution of $\Bbbk$ over $\Bbbk[v_1,\ldots,v_m]$.

\begin{definition}
A simplicial complex $K$ is called Golod over $\Bbbk$ if all products and (higher) Massey products in $\mathrm{Tor}_{\Bbbk[v_1,\ldots,v_m]}^*(\Bbbk[K],\Bbbk)$ vanish.
\end{definition}

\begin{remark}
\label{gcd-gap}
Recently Berglund and J\"ollenbeck \cite{BJ} published a paper in which they claimed that the condition on (higher) Massey products is redundant. The result heavily depends on the paper of J\"ollenbeck \cite{Jo} in which several gaps are found. So we do not remove the condition on (higher) Massey products from the definition of the Golodness.
\end{remark}

There is a combinatorial description of products in $\mathrm{Tor}_{\Bbbk[v_1,\ldots,v_m]}^*(\Bbbk[K],\Bbbk)$ due to Hochster (cf. \cite{S}), and we recall it here. We start with an isomorphism
$$\mathrm{Tor}_{\Bbbk[v_1,\ldots,v_m]}^i(\Bbbk[K],\Bbbk)\cong\bigoplus_{I\subset[m]}\widetilde{H}^{i-|I|-1}(K_I;\Bbbk)$$
shown by Hochster, which can be deduced also from the BBCG decomposition. Through this isomorphism, the products in $\mathrm{Tor}_{\Bbbk[v_1,\ldots,v_m]}^*(\Bbbk[K],\Bbbk)$ split into maps
$$\widetilde{H}^{i-|I|-1}(K_I;\Bbbk)\otimes\widetilde{H}^{j-|J|-1}(K_J;\Bbbk)\to\widetilde{H}^{i+j-|I|-|J|-1}(K_{I\cup J};\Bbbk)$$
for $I,J\subset[m]$. Hochster showed that this map is trivial for $I\cap J\ne\emptyset$ and is induced from the inclusion $K_{I\cup J}\to K_I*K_J$ for $I\cap J=\emptyset$. Then one can naively define a notion of simplicial complexes which implies the Golodness.

\begin{definition}
A simplicial complex $K$ is (resp. stably) homotopy Golod if the inclusion $K_{I\cup J}\to K_I*K_J$ is (resp. stably) null homotopic for all $\emptyset\ne I,J\subset[m]$ satisfying $I\cap J=\emptyset$, and all (higher) Massey products in $\mathrm{Tor}_{\Bbbk[v_1,\ldots,v_m]}^*(\Bbbk[K],\Bbbk)$ are trivial for any $\Bbbk$.
\end{definition}

By definition, we obviously have:

\begin{proposition}
\label{Golod-implication}
The homotopy Golodness implies the stable homotopy Golodness, and the stable homotopy Golodness implies the Golodness over any ring.
\end{proposition}

We next consider a connection between Stanley-Reisner rings and polyhedral products. By definition one immediately sees that there is an isomorphism
$$H^*(\Z_K(\mathbb{C}P^\infty,*);\Bbbk)\cong\Bbbk[K].$$
This was first found by Davis and Januszkiewicz \cite{DJ}, and since then, the combinatorial aspect of polyhedral products has been studied extensively in connection with Stanley-Reisner rings. By the above isomorphism several derived algebras of the Stanley-Reisner ring $\Bbbk[K]$ can be realized by the cohomology of spaces related with polyhedral products. In particular, there is a ring isomorphism
\begin{equation}
\label{Tor}
H^*(\Z_K;\Bbbk)\cong\mathrm{Tor}_{\Bbbk[v_1,\ldots,v_m]}^*(\Bbbk[K],\Bbbk)
\end{equation}
which was proved by Baskakov, Buchstaber, and Panov \cite{BBP}. (This isomorphism is actually induced from a chain homotopy equivalence between the cellular cochain complex of $\Z_K$ and the Koszul resolution of $\Bbbk$ over $\Bbbk[v_1,\ldots,v_m]$ tensored with $\Bbbk[K]$.) If the BBCG decomposition of $\Z_K(C\underline{X},\underline{X})$ desuspends, then $\Z_K(C\underline{X},\underline{X})$ becomes a suspension, so we obtain:

\begin{proposition}
\label{BBCG-Golod}
If the BBCG decomposition of $\Z_K$ desuspends, $K$ is Golod over any ring.
\end{proposition}

So by Theorem \ref{main-decomp} and \ref{main-decomp-Z}, the triviality of the fat wedge filtrations of $\RZ_K$ and $\Z_K$ seems too strong to guarantee the Golodness of $K$. Then it is natural to seek for structures connected with the triviality of the fat wedge filtration of $\RZ_K$ and $\Z_K$ which are higher than the Golodness of $K$. For the rest of this section we prove the higher structures that we seek for are the (stable) homotopy Golodness.

We first relate the attaching maps $\varphi_{K_I}$ to Whitehead products. For simplicial complexes $L_1,L_2$ with disjoint vertex sets, we have
$$\RZ_{L_1*L_2}=\RZ_{L_1}\times\RZ_{L_2}$$
by Example \ref{non-desuspension-example}. Then in particular we get
$$\RZ_{L_1*L_2}^{m_1+m_2-1}=(\RZ_{L_1}^{m_1-1}\times\RZ_{L_2})\cup(\RZ_{L_1}\times\RZ_{L_2}^{m_2-1})$$
where $m_1,m_2$ are the numbers of vertices of $L_1,L_2$ respectively. So there is a projection
$$\RZ_{L_1*L_2}^{m_1+m_2-1}\to\RZ_{L_1}/\RZ_{L_1}^{m_1-1}\vee\RZ_{L_2}/\RZ_{L_2}^{m_2-1}=|\Sigma\Sd L_1|\vee|\Sigma\Sd L_2|$$
where the last equality holds by Theorem \ref{cone-decomp}.

\begin{proposition}
\label{Whitehead-product}
For $\emptyset\ne I,J\subset[m]$ satisfying $I\cap J=\emptyset$, the composite
$$|\Sd(K_I*K_J)|\xrightarrow{\varphi_{K_I*K_J}}\RZ_{K_I*K_J}^{|I|+|J|-1}\xrightarrow{\rm proj}|\Sigma\Sd K_I|\vee|\Sigma\Sd K_J|$$
is identified with the Whitehead product.
\end{proposition}

\begin{proof}
By Theorem \ref{cone-decomp}, we have $\RZ_L=\RZ_L^{\ell-1}\cup_{\varphi_L}C|(\Sd L)|$ for any simplicial complex $L$ with $\ell$ vertices. Then by 
$$|\Sd(K_I*K_J)|=|\Sd K_I|*|\Sd K_J|=(|\Sd K_I|\times C|\Sd K_J|)\cup(C|\Sd K_I|\times|\Sd K_J|)$$
and the definition of $\varphi_{K_I*K_J}$, the map $\varphi_{K_I*K_J}$ is identified with the map
\begin{multline*}
(\varphi_{K_I}\times C\varphi_{K_J})\cup(C\varphi_{K_I}\times\varphi_{K_J})\colon(|\Sd K_I|\times C|\Sd K_J|)\cup(C|\Sd K_I|\times|\Sd K_J|)\\
\to(\RZ_{K_I}^{|I|-1}\times(\RZ_{K_J}^{|J|-1}\cup C|\Sd K_J|))\cup((\RZ_{K_I}^{|I|-1}\cup C|\Sd K_I|)\times\RZ_{K_J}^{|J|-1}).
\end{multline*}
Thus the proposition follows from an easy inspection.
\end{proof}

We now connect the triviality of the fat wedge filtration of $\RZ_K$ to the homotopy Golodness of $K$. To this end we prepare a small lemma which immediately follows from the Hilton-Milnor theorem when spaces are path-connected but we are now considering disconnected spaces.

\begin{lemma}
\label{Whitehead-inj}
Let $X,Y$ be CW-complexes, not necessarily connected, and let $w\colon X*Y\to\Sigma X\vee\Sigma Y$ denote the Whitehead product. For any map $f\colon A\to X*Y$, if $w\circ f$ is null homotopic, then  so is $f$.
\end{lemma}

\begin{proof}
Let $F$ be the homotopy fiber of the Whitehead product $w$. It is sufficient to show that the fiber inclusion $F\to X*Y$ is null homotopic. Consider the homotopy fibration sequence
$$\Omega(X*Y)\xrightarrow{\Omega w}\Omega(\Sigma X\vee\Sigma Y)\to F\to X*Y.$$
If there is a left homotopy inverse of $\Omega w$, we get a homotopy equivalence $\Omega(\Sigma X\vee\Sigma Y)\simeq\Omega(X*Y)\times F$ hence a right homotopy inverse of the map $\Omega(\Sigma X\vee\Sigma Y)\to F$. Then the fiber inclusion $F\to X*Y$ is null homotopic. So we construct a left homotopy inverse of $\Omega w$. Consider the inclusion $j\colon\Sigma X\vee\Sigma Y\to\Sigma X\times\Sigma Y$. Then its homotopy fiber is homotopy equivalent to $\Omega\Sigma X*\Omega\Sigma Y$, and the fiber inclusion $\Omega\Sigma X*\Omega\Sigma Y\to\Sigma X\vee\Sigma Y$ is the Whitehead product of the evaluation maps $\Sigma\Omega\Sigma X\to\Sigma X$ and $\Sigma\Omega\Sigma Y\to\Sigma Y$, which we denote by $\bar{w}$. Since $\Omega j$ has a right homotopy inverse, $\Omega\bar{w}$ admits a left homotopy inverse. Then since $w=\bar{w}\circ(E*E)$ for the suspension map $E\colon A\to\Omega\Sigma A$ and $E*E$ has a left homotopy inverse, $\Omega w$ admits a left homotopy inverse. Therefore the proof is done.
\end{proof}

\begin{theorem}
\label{main-Golod}
The following hold:
\begin{enumerate}
\item If the fat wedge filtration of $\RZ_K$ is trivial, then $K$ is homotopy Golod;
\item If the fat wedge filtration of $\Z_K$ is trivial, then $K$ is stably homotopy Golod.
\end{enumerate}
\end{theorem}

\begin{proof}
(1) For $\emptyset\ne I,J\subset[m]$ with $I\cap J=\emptyset$, there is a commutative diagram
$$\xymatrix{|\Sd K_{I\cup J}|\ar[rr]^{\varphi_{K_{I\cup J}}}\ar[d]^{\rm incl}&&\RZ_{K_{I\cup J}}^{|I|+|J|-1}\ar[d]^{\rm incl}\\
|\Sd(K_I*K_J)|\ar[rr]^{\varphi_{K_I*K_J}}&&\RZ_{K_I*K_J}^{|I|+|J|-1}.}$$
If the map $\varphi_{K_{I\cup J}}$ is null homotopic, so is the composite
$$|\Sd K_{I\cup J}|\xrightarrow{\rm incl}|\Sd(K_I*K_J)|\to|\Sigma\Sd K_I|\vee|\Sigma\Sd K_J|$$
where the last arrow is the Whitehead product by Proposition \ref{Whitehead-product}. Then by Lemma \ref{Whitehead-inj}, the inclusion $|K_{I\cup J}|\to|K_I*K_J|$ is null homotopic. By Theorem \ref{main-decomp}, $\Z_K$ is a suspension, implying that all (higher) Massey products in its cohomology are trivial. Thus the proof is completed.

(2) By the Freudenthal suspension theorem one sees that $K$ is stably homotopy Golod if and only if the inclusion $\Sigma^{|I|+|J|}|K_{I\cup J}|\to\Sigma^{|I|}|K_I|*\Sigma^{|J|}|K_J|$ is null homotopic for any $\emptyset\ne I,J\subset[m]$ with $I\cap J=\emptyset$. By the quite same way as the above proof, we can deduce that the triviality of the fat wedge filtration of $\Z_K$ implies the latter condition.
\end{proof}


\section{Homology fillable complexes}

In this section, we introduce a new class of simplicial complexes which we call (homology) fillable complexes, and prove that the fat wedge filtration of $\RZ_K$ is trivial if $K$ is (homology) fillable, where homology fillable complexes are a homological generalization of fillable complexes. 

\subsection{Fillable complexes}

We first consider fillable complexes. Recall that a subset $M\subset[m]$ is a minimal non-face of $K$ if $M$ is not a simplex of $K$ but $M-v$ is a simplex of $K$ for each $v\in M$. Notice that if $M$ is a minimal non-face of $K$, then $K\cup M$ is a simplicial complex.

\begin{definition}
\label{def-fillable}
A simplicial complex $K$ is fillable if there are minimal non-faces $M_1,\ldots,M_r$ of $K$ such that $|K\cup M_1\cup\cdots\cup M_r|$ is contractible.
\end{definition}

We here remark that in Definition \ref{def-fillable}, $K\cup M_1\cup\cdots\cup M_i$ is a simplicial complex and $M_{i+1},\ldots,M_r$ are its minimal non-faces for all $i$.

\begin{theorem}
\label{fillable-1}
If $K$ is fillable, then $\varphi_K$ is null homotopic.
\end{theorem}

\begin{proof}
We observe how the attaching map $\varphi_K$ behaves with minimal non-faces of $K$. Let $M\subset[m]$ be a minimal non-face of $K$. Then by the definition of the embedding $i_c\colon|\Sd\Delta^{[m]}|\to(D^1)^{\times m}$, we have
$$i_c(|\Sd\Delta^M|)=\bigcup_{\emptyset\ne\sigma\subset M}C_{\sigma\subset M}.$$
Since $M$ is a minimal non-face of $K$, $M-\sigma$ is a simplex of $K$ for any $\emptyset\ne\sigma\subset M$. Then by \eqref{RZ-C} we obtain
$$i_c(|\Sd\Delta^M|)\subset\RZ_K^{m-1}.$$
Since $K$ is fillable, there are minimal non-faces $M_1,\ldots,M_r$ such that $|K\cup M_1\cup\cdots\cup M_r|$ is contractible. By the above observation, the attaching map $\varphi_K\colon|\Sd K|\to\RZ_K^{m-1}$ factors as
$$|\Sd K|\xrightarrow{\rm incl}|\Sd(K\cup M_1\cup\cdots\cup M_r)|\xrightarrow{i_c}\RZ_K^{m-1}.$$
Thus $\varphi_K$ is null homotopic.
\end{proof}

We immediately obtain the following.

\begin{corollary}
\label{fillable-2}
If $K_I$ is fillable for all $\emptyset\ne I\subset[m]$, the fat wedge filtration of $\RZ_K$ is trivial.
\end{corollary}

By Theorem \ref{main-decomp} and Corollary \ref{fillable-2}, the BBCG decomposition desuspends for simplicial complexes whose full subcomplexes are fillable, so we get a description of the homotopy types of the corresponding polyhedral products. In order to get a more complete description of the homotopy types, we determine the homotopy type of $|\Sigma K|$ when $K$ is fillable.

\begin{proposition}
\label{fillable-sphere}
If $K$ is fillable, then $|\Sigma K|$ is homotopy equivalent to a wedge of spheres.
\end{proposition}

\begin{proof}
Since $K$ is fillable, $|K\cup M_1\cup\cdots\cup M_r|$ becomes contractible for some minimal non-faces $M_1,\ldots,M_r$ of $K$. Then there is a homotopy equivalence $|\Sigma K|\simeq|K\cup M_1\cup\cdots\cup M_r|/|K|$ because $|K\cup M_1\cup\cdots\cup M_r|$ is contractible and $(|K\cup M_1\cup\cdots\cup M_r|,|K|)$ is an NDR pair, where $|K\cup M_1\cup\cdots\cup M_r|/|K|$ is a wedge of spheres. 
\end{proof}

\begin{corollary}
\label{fillable-decomp2}
If $K_I$ is fillable for any $\emptyset\ne I\subset[m]$, then $\Z_K(D^n,S^{n-1})$ is homotopy equivalent to a wedge of spheres.
\end{corollary}

\begin{proof}
Combine Theorem \ref{main-decomp}, Corollary \ref{fillable-2}, and Proposition \ref{fillable-sphere}.
\end{proof}

\subsection{Homology fillable complexes}

We next consider a homological generalization of fillability. Recall that $K$ is $i$-acyclic over $\Bbbk$ if $\widetilde{H}_*(K;\Bbbk)=0$ for $*\le i$. If $K$ is $i$-acyclic over $\Bbbk$ for any $i$, $K$ is called acyclic over $\Bbbk$.

\begin{definition}
A simplicial complex $K$ is homology fillable if for each connected component $L$ of $K$ and any prime $p$,
\begin{enumerate}
\item there are minimal non-faces $M_1^p,\ldots,M_{r_p}^p$ of $L$ such that $L\cup M_1^p\cup\cdots\cup M_{r_p}^p$ is acyclic over $\mathbb{Z}/p$, and
\item $|\widehat{L}|$ is simply connected, where $\widehat{L}$ is obtained from $L$ by adding all minimal non-faces.
\end{enumerate}
\end{definition}

Roughly, the first condition of the above definition corresponds to the component-wise fillability at the prime $p$, and the second condition guarantees the integrability of this local contractibility. We integrate the $p$-local results by the following lemma.

\begin{lemma}
[Bousfield and Kan {\cite[fracture square lemma 6.3]{BK}}]
\label{fracture-lemma}
Let $X$ be a connected finite CW-complex and let $Y$ be a connected nilpotent CW-complex of finite type. If maps $f,g\colon X\to Y$ satisfy $f_{(p)}\simeq g_{(p)}$ for all prime $p$, then $f\simeq g$.
\end{lemma}

\begin{theorem}
\label{homology-fillable-phi}
If $K$ is homology fillable, then $\varphi_K$ is null homotopic.
\end{theorem}

\begin{proof}
We have observed in the proof of Theorem \ref{fillable-1} that the attaching map $\varphi_K\colon|\Sd K|\to\RZ_K^{m-1}$ factors through the inclusion
$$|\Sd K|=|\Sd K_1|\sqcup\cdots\sqcup|\Sd K_s|\to|\Sd\widehat{K}_1|\sqcup\cdots\sqcup|\Sd\widehat{K}_s|$$
where $K_1,\ldots,K_s$ are the connected components of $K$. Then since $\RZ_K^{m-1}$ is connected, it is sufficient to show that the inclusion $|K_i|\to|\widehat{K}_i|$ is null homotopic for $i=1,\ldots,s$. Since $K$ is homology fillable, there are minimal non-faces $M_1^{i,p},\ldots,M_{r(i,p)}^{i,p}$ such that $K_i\cup M_1^{i,p}\cup\cdots\cup M_{r(i,p)}^{i,p}$ is acyclic over $\mathbb{Z}/p$. Since $K_i\cup M_1^{i,p}\cup\cdots\cup M_{r(i,p)}^{i,p}$ is of finite type, its acyclicity over $\mathbb{Z}/p$ implies the acyclicity over $\mathbb{Z}_{(p)}$, so its $p$-localization is contractible. Then the $p$-localization of the inclusion $|K_i|\to|\widehat{K}_i|$ is null homotopic since it factors through $|K_i\cup M_1^{i,p}\cup\cdots\cup M_{r(i,p)}^{i,p}|$. Then by Lemma \ref{fracture-lemma} and the assumption that $|\widehat{K}_i|$ is simply connected, we obtain that the inclusion itself is null homotopic, completing the proof.
\end{proof}

\begin{corollary}
\label{homology-fillable-decomp}
If $K_I$ is homology fillable for any $\emptyset\ne I\subset[m]$, then the fat wedge filtration of $\RZ_K$ is trivial.
\end{corollary}

So we obtain the decomposition of polyhedral products and the homotopy Golodness of simplicial complexes whose full subcomplexes are homology fillable by Theorem \ref{main-decomp} and \ref{main-Golod}. As well as fillable complexes, we can determine the homotopy type of a suspension of a homology fillable complex. We prepare a technical lemma.

\begin{lemma}
\label{wedge-localization}
Let $X$ be a connected CW-complex of finite type. If $\Sigma X$ has the $p$-local homotopy type of a wedge of spheres for all prime $p$, then $\Sigma X$ itself has the homotopy type of a wedge of spheres.
\end{lemma}

\begin{proof}
By assumption, $H_i(\Sigma X;\mathbb{Z})$ is a free abelian group of finite rank for each $i$. Choose a basis $x^1_i,\ldots,x^{n_i}_i$ of $H_i(\Sigma X;\mathbb{Z})$ for $i>0$. Using a $p$-local homotopy equivalence between $\Sigma X$ and a wedge of spheres, we can easily construct a map ${}_p\theta^j_i\colon S^i\to\Sigma X_{(p)}$ satisfying $({}_p\theta^j_i)_*(u_i)=x^j_i$ in homology with coefficient $\mathbb{Z}_{(p)}$ for any $i>0$ and $j=1,\ldots,n_i$, where $u_i$ is a generator of $H_i(S^i;\mathbb{Z})\cong\mathbb{Z}$. Let $\{p_1,p_2,\ldots\}$ be the set of all primes except for $p$. It is well-known that the $p$-localization $\Sigma X_{(p)}$ is given by the homotopy colimit of the sequence of maps
$$\Sigma X\xrightarrow{\underline{l_1}}\Sigma X\xrightarrow{\underline{l_2}}\Sigma X\xrightarrow{\underline{l_3}}\Sigma X\xrightarrow{\underline{l_4}}\cdots$$
where $l_k=p_1\cdots p_k$ and $\underline{q}\colon\Sigma X\to\Sigma X$ is the degree $q$ map.
By the compactness of $S^i$, ${}_p\theta^j_i$ factors through the finite step of the above sequence. Then there is a map ${}_p\bar{\theta}^j_i\colon S^i\to\Sigma X$ satisfying $({}_p\bar{\theta}^j_i)_*(u_i)={}_pa^j_ix^j_i$ for $a^j_i\in\mathbb{Z}$ with $p\nmid {}_pa^j_i$ in integral homology. Now we can choose primes $q_1,\ldots,q_n$ such that ${}_{q_1}a^j_i,\ldots,{}_{q_n}a^j_i$ are relatively prime, so there are integers $d_1,\ldots,d_n$ satisfying $d_1({}_{q_1}a^j_i)+\ldots+d_n({}_{q_n}a^j_i)=1$. Then the map
$$\lambda^j_i:=\underline{d_1}\circ{}_{q_1}\bar{\theta}^j_i+\cdots+\underline{d_n}\circ{}_{q_n}\bar{\theta}^j_i$$
satisfies $(\lambda^j_i)_*(u_i)=x^j_i$ in integral homology, where
 the sum is defined by using the suspension comultiplication of $\Sigma X$. Thus the map $\bigvee_{i\ge 1}\bigvee_{j=1}^{n_i}\lambda^j_i\colon\bigvee_{i\ge 1}\bigvee_{j=1}^{n_i}S^i\to\Sigma X$ is an isomorphism in integral homology, hence a homotopy equivalence by the J.H.C. Whitehead theorem, where $\Sigma X$ is simply connected since $X$ is connected. Therefore the proof is completed.
\end{proof}

\begin{proposition}
\label{homology-fillable-sphere}
If $K$ is homology fillable, $|\Sigma K|$ is homotopy equivalent to a wedge of spheres.
\end{proposition}

\begin{proof}
It is sufficient to consider the case that $K$ is connected. Let $M_1^p,\ldots,M_r^p$ be minimal non-faces of $K$ such that $K\cup M_1^p\cup\cdots\cup M_r^p$ is acyclic over $\mathbb{Z}/p$. Since $|K\cup M_1^p\cup\cdots\cup M_r^p|$ is a finite complex, it is also acyclic over $\mathbb{Z}_{(p)}$, so its $p$-localization is contractible. Then as in the proof of Proposition \ref{fillable-sphere}, $|\Sigma K|_{(p)}$ is homotopy equivalent to $(|K\cup M_1^p\cup\cdots\cup M_r^p|/|K|)_{(p)}$ which is a wedge of $p$-local spheres. (Note that the dimension of each sphere in the wedge is greater than 1 since $K$ is connected. So we can commute the localization and the wedge.) Thus the proof is completed by Lemma \ref{wedge-localization}.
\end{proof}

\begin{corollary}
\label{Z-homology-fillable-sphere}
If $K_I$ is homology fillable for all $\emptyset\ne I\subset[m]$, then $\Z_K(D^n,S^{n-1})$ is homotopy equivalent to a wedge of spheres.
\end{corollary}

\begin{proof}
Combine Theorem \ref{main-decomp} and Corollary \ref{homology-fillable-decomp}, and Proposition \ref{homology-fillable-sphere}.
\end{proof}


\section{Shellable and sequentially Cohen-Macaulay comlplexes}

In this section, we show that the fat wedge filtrations of the real moment-angle complexes for dual shellable and dual sequentially Cohen-Macaulay complexes are trivial by proving their fillability and homology fillability, where the definitions of shellable and sequentially Cohen-Macaulay complexes will be given in Definition \ref{def-shellable} and \ref{def-SCM}. Our choice of dual shellable and dual sequentially Cohen-Macaulay complexes are motivated by the following, where dual shellable complexes are dual sequentially Cohen-Macaulay complexes as in \eqref{SCM-implication}, and the easier cases of shifted and dual vertex-decomposable complexes were studied in \cite{GT1,IK1,GW}.

\begin{proposition}
[Herzog, Reiner, and Welker \cite{HRW}]
\label{Golod}
The Alexander duals of sequentially Cohen-Macaulay complexes over $\Bbbk$ are Golod over $\Bbbk$.
\end{proposition}

We first consider the case of dual shellable complexes, and next generalize the arguments for dual shellable complexes homologically for dual sequentially Cohen-Macaulay complexes.

\subsection{Shellable complex}

We first recall the definition of shellable complexes from \cite{BW}, where shellability is one of the most active subject studied in combinatorics. Maximal simplices of a simplicial complex are called facets, and if all facets have the same dimension, then the simplicial complex is called pure.

\begin{definition}
\label{def-shellable}
A simplicial complex $K$ is shellable if there is an ordering of facets $F_1,\ldots,F_t$, called a shelling, such that
$$\langle F_k\rangle\cap\langle F_1,\ldots,F_{k-1}\rangle$$
is pure and $(|F_k|-2)$-dimensional for $k=2,\ldots,t$, where $\langle F_1,\ldots,F_{k-1}\rangle$ means a subcomplex of $K$ generated by $F_1,\ldots,F_{k-1}$.
\end{definition}

Interesting examples of shellable complexes can be found in \cite{BW,H}. We next recall the Alexander dual of a simplicial complex.

\begin{definition}
Let $L$ be a simplicial complex whose vertex set is a subset of a finite set $S$. The Alexander dual of $L$ with respect to $S$ is defined by
$$L^\vee:=\{\sigma\subset S\,\vert\, S-\sigma\not\in L\}.$$
\end{definition}

Of course the Alexander dual of $L$ changes if we alter the ambient set $S$, so we must be careful for the ambient set to take the Alexander dual. The Alexander dual of $K$ and $\dl_K(v)$ for $v\in[m]$ will be always taken over $[m]$ and $[m]-v$, respectively. It is easy to verify
$$(L^\vee)^\vee=L$$
where the duals of $L$ and $L^\vee$ are taken over $S$. As well as the topological Alexander dual, the duality of (co)homology holds for the Alexander dual of a simplicial complex.

\begin{theorem}
[cf. \cite{BT}] 
\label{duality}
Let $L$ be a simplicial complex whose vertex set is a subset of a finite set $S$. Then for any $i$,
$$\widetilde{H}_i(L;\Bbbk)\cong\widetilde{H}^{|S|-i-3}(L^\vee;\Bbbk)$$
where the Alexander dual of $L$ is taken over $S$.
\end{theorem}

The following properties of the Alexander duals will play a fundamental role in showing the fillability of dual shellable complexes.

\begin{lemma}
\label{duality-dictionary}
The following hold:
\begin{enumerate}
\item $F$ is a facet of $K^\vee$ if and only if $F^\vee:=[m]-F$ is a minimal non-face of $K$;
\item $\dl_K(v)^\vee=\lk_{K^\vee}(v)$ for any $v\in[m]$.
\end{enumerate}
\end{lemma}

\begin{proof}
(1) $F$ is a facet of $K^\vee$ if and only if $F^\vee\not\in K$ and $(F\cup v)^\vee\in K$ for any $v\not\in F$. Since $F^\vee-v=[m]-(F\cup v)=(F\cup v)^\vee$ for any $v\in F^\vee$, the proof is done.

(2) For any $v\in[m]$, we have
\begin{alignat*}{3}
\dl_K(v)^\vee&=\{\sigma\subset[m]-v\,\vert\,([m]-v)-\sigma\not\in\dl_K(v)\}&&=\{\sigma\subset[m]-v\,\vert\,[m]-(\sigma\cup v)\not\in K\}\\
&=\{\tau\in K^\vee\,\vert v\not\in\tau\text{ and }\tau\cup v\in K^\vee\}&&=\lk_{K^\vee}(v).
\end{alignat*}
\end{proof}

We show that the dual shellability is preserved by a vertex deletion. 

\begin{lemma}
[Bj\"orner and Wachs {\cite[Proposition 10.14]{BW}}]
\label{shellable-lk}
For a shellable complex $L$ and its vertex $v$, the link $\lk_L(v)$ is shellable.
\end{lemma}

\begin{proof}
Let $F_1,\ldots,F_t$ be a shelling of $L$ such that $F_{i_1},\ldots,F_{i_r}$ are all facets including the vertex $v$ with $i_1<\cdots<i_r$. Put $G_k=F_{i_k}-v$. Then $G_1,\ldots,G_r$ are all facets of $\lk_L(v)$. Since $F_1,\ldots,F_t$ is a shelling of $L$, there exists $j<k$ for $k=2,\ldots,t$ and $w\in F_k$, such that $F_k-w\subset F_j$, implying that $G_1,\ldots,G_r$ is a shelling of $\lk_L(v)$. 
\end{proof}

\begin{proposition}
\label{shellable-dl}
If $K^\vee$ is shellable, then so is $\dl_K(v)^\vee$ for any $v\in[m]$.
\end{proposition}

\begin{proof}
If $[m]-v$ is a simplex of $K$, $\dl_K(v)^\vee$ is trivially shellable. Then we may assume that $[m]-v$ is not a simplex of $K$, or equivalently $v$ is a vertex of $K^\vee$. Thus the proof is done by combining Lemma \ref{duality-dictionary} and \ref{shellable-lk}.
\end{proof}

We next show the fillability of dual shellable complexes. 

\begin{lemma}
\label{collapsible}
If the Alexander dual of $K$ is collapsible, then $|K|$ is contractible.
\end{lemma}

\begin{proof}
Suppose that for $\sigma\subset\tau\subset[m]$, $\tau^\vee$ is a free face of $K^\vee$ such that $\sigma^\vee$ is the only simplex of $K^\vee$ satisfying $\tau^\vee\subset\sigma^\vee$ and $\dim\sigma^\vee=\dim\tau^\vee+1$. Then by inspection, one deduces that $K\cup\{\sigma,\tau\}$ is a simplicial complex and $\sigma$ is a free face of $K\cup\{\sigma,\tau\}$ such that $\tau$ is the only simplex of $K\cup\{\sigma,\tau\}$ satisfying $\sigma\subset\tau$ with $\dim\tau=\dim\sigma+1$. In particular $|K|$ and $|K\cup\{\sigma,\tau\}|$ have the same homotopy type. On the other hand, we have
$$(K\cup\{\sigma,\tau\})^\vee=K^\vee-\{\sigma^\vee,\tau^\vee\},$$
where the right hand side is the elementary collapse of $K^\vee$ with respect to the free face $\tau^\vee$. Then since $K^\vee$ is collapsible, by iterating the above procedure, we see that $|K|$ is homotopy equivalent to the Alexander dual of the 0-simplex $\Delta^{\{v\}}$ for some $v\in[m]$, where the dual of $\Delta^{\{v\}}$ is taken over $[m]$. This Alexander dual of $\Delta^{\{v\}}$ is obviously the star of the vertex $v$ in $\Delta^{[m]}$ which is contractible. Therefore the proof is completed.
\end{proof}

\begin{proposition}
\label{shellable-fillable}
If the Alexander dual of $K$ is shellable, then $K$ is fillable.
\end{proposition}

\begin{proof}
We first note that $K-F$ is a subcomplex of $K$ whenever $F$ is a facet of $K$. Let $F_1,\ldots,F_t$ be a shelling of $K^\vee$ such that $F_{i_1},\ldots,F_{i_r}$ are all facets of $K^\vee$ satisfying $\langle F_{i_k}\rangle\cap\langle F_1,\ldots,F_{i_k-1}\rangle=\partial F_{i_k}$, which are called the spanning facets. Then one immediately sees that $K^\vee-\{F_{i_1},\ldots,F_{i_r}\}$ is collapsible. By Lemma \ref{duality-dictionary}, each $F_i^\vee$ is a minimal non-face of $K$, implying that $K\cup\{F_{i_1}^\vee,\ldots,F_{i_r}^\vee\}$ is a simplicial complex. Since 
$$(K\cup\{F_{i_1}^\vee,\ldots,F_{i_r}^\vee\})^\vee=K^\vee-\{F_{i_1},\ldots,F_{i_r}\},$$
it follows from Lemma \ref{collapsible} that $|K\cup\{F_{i_1}^\vee,\ldots,F_{i_r}^\vee\}|$ is contractible, completing the proof.
\end{proof}

We now obtain:

\begin{theorem}
\label{shellable-fillable-sub}
If the Alexander dual of $K$ is shellable, then $K_I$ is fillable for any $\emptyset\ne I\subset[m]$.
\end{theorem}

\begin{proof}
Combine Proposition \ref{shellable-dl} and \ref{shellable-fillable}, where every full subcomplex is obtained by consecutive vertex deletions.
\end{proof}

\begin{corollary}
\label{shellable-trivial}
If the Alexander dual of $K$ is shellable, the fat wedge filtration of $\RZ_K$ is trivial.
\end{corollary}

\begin{proof}
Combine Corollary \ref{fillable-2} and Theorem \ref{shellable-fillable-sub}.
\end{proof}

\begin{corollary}
If the Alexander dual of $K$ is shellable, then $\Z_K(D^n,S^{n-1})$ has the homotopy type of a wedge of spheres.
\end{corollary}

\begin{proof}
Combine Theorem \ref{main-decomp}, Proposition \ref{fillable-sphere}, and Corollary \ref{shellable-trivial}.
\end{proof}

\subsection{Sequentially Cohen-Macaulay complex}

Recall that a simplicial complex $K$ is Cohen-Macaulay (CM, for short) over a ring $\Bbbk$ if its Stanley-Reisner ring $\Bbbk[K]$ is a Cohen-Macaulay ring, that is, the Krull dimension and the depth of $\Bbbk[K]$ are the same. By definition CM complexes are pure, and sequentially Cohen-Macaulay (SCM, for short) complexes were introduced as a non-pure generalization of CM complexes \cite{S}. 

\begin{definition}
\label{def-SCM}
A simplicial complex $K$ is sequentially Cohen-Macaulay over $\Bbbk$ if the subcomplex of $K$ generated by $i$-dimensional faces is Cohen-Macaulay over $\Bbbk$ for $i\ge 0$.
\end{definition}

 By definition, we have
 $$\text{pure and SCM over }\Bbbk\quad\Longleftrightarrow\quad\text{CM over }\Bbbk.$$
 As well as CM complexes, there is a useful homological characterization of SCM complexes. For a simplicial complex $L$ and $i\ge 0$, let $L^{\langle i\rangle}$ denote the subcomplex of $L$ generated by faces of dimension $\ge i$. 

\begin{proposition}
[Bj\"orner and Wachs \cite{BW}]
\label{SCM-def}
A simplicial complex $K$ is SCM over $\Bbbk$ if and only if for any $\sigma\in K$ and $i\ge 0$, $\lk_K(\sigma)^{\langle i\rangle}$ is $(i-1)$-acyclic over $\Bbbk$
\end{proposition}

Pure shellability of simplicial complexes were introduced as a combinatorial criterion for CMness, and we also have an implication \eqref{SCM-implication} in the non-pure case. We now start to show all full subcomplexes of a dual SCM complex over $\mathbb{Z}$ are homology fillable by generalizing the above arguments for dual shellable complexes. The key is the following homological generalization of spanning facets which play the important role in the proof of Proposition \ref{shellable-fillable}. Facets $F_1,\ldots,F_r$ of a simplicial complex $L$ are called homology spanning facets over $\Bbbk$ if $L-\{F_1,\ldots,F_r\}$ is acyclic over $\Bbbk$. Let us search for homology spanning facets of SCM complexes.

\begin{lemma}
\label{cycle-facet}
Let $\Bbbk$ be a field and $L$ be a simplicial complex satisfying $\widetilde{H}_i(L^{\langle i+1\rangle};\Bbbk)=0$. Then any non-boundary $i$-cycle of $L$ over $\Bbbk$ involves a facet of dimension $i$. 
\end{lemma}

\begin{proof}
Let $x$ be an $i$-cycle of $L$ over $\Bbbk$. If $x$ involves no facet of dimension $i$, it is a cycle of $L^{\langle i+1\rangle}$ over $\Bbbk$. Then since $\widetilde{H}_i(L^{\langle i+1\rangle};\Bbbk)=0$, $x$ is a boundary, completing the proof.
\end{proof}

\begin{proposition}
\label{homology-spanning}
If $L$ is an SCM complex over a field $\Bbbk$, then it has homology spanning facets over $\Bbbk$.
\end{proposition}

\begin{proof}
Choose a basis $x_i^1,\ldots,x_i^{n_i}$ of $\widetilde{H}_i(L;\Bbbk)$ for $i\ge 0$. By Lemma \ref{cycle-facet}, $x_i^1$ involves a facet $F_i^1$, and by subtracting a multiple of $x_i^1$ from $x_i^2,\ldots,x_i^{n_i}$ if necessary, we may assume that $x_i^2,\ldots,x_i^{n_i}$ do not involve $F_i^1$. Then by induction, we see that for $j=1,\ldots,n_i$ and $k\ne j$, $x_i^j$ involves a facet $F_i^j$ and $x_i^k$ does not involve a facet $F_i^j$. We shall show that facets $F_0^1,\ldots,F_0^{n_0},\ldots,F_d^1,\ldots,F_d^{n_d}$ are homology spanning facets of $L$ over $\Bbbk$, where $d=\dim L$. Put $\Delta=L-\{F_0^1,\ldots,F_0^{n_0},\ldots,F_d^1,\ldots,F_d^{n_d}\}$. 
Then we have
$$|L|/|\Delta|=\bigvee_{i=0}^d\bigvee_{j=1}^{n_i}|F_i^j|/|\partial F_i^j|=\bigvee_{i=0}^d\bigvee_{j=1}^{n_i}S^i_j$$
where $S^i_j$ is a copy of $S^i$. Note that the projection $|L|\to |L|/|\Delta|$ sends $x_i^j$ to a generator of $H_i(S^i_j;\Bbbk)$. Then this projection is an isomorphism in homology with coefficient $\Bbbk$, hence the proof is completed by the Puppe exact sequence of the homotopy cofibration $|\Delta|\to|L|\to|L|/|\Delta|$.
\end{proof}

Regarding the second condition of homology fillability, we prove the following. Recall from Section 7 that for a connected simplicial complex $L$, the simplicial complex $\widehat{L}$ is defined by adding all minimal non-faces to $L$.

\begin{proposition}
\label{Golod-1-conn}
If a connected simplicial complex $L$ is Golod over some ring $\Bbbk$, then $|\widehat{L}|$ is simply connected.
\end{proposition}

\begin{proof}
If there is a minimal cycle in $L$ of length $\ge 4$, say $C$, then $C$ is a full subcomplex of $L$, hence $\Z_C$ is a retract of $\Z_L$. It follows from \cite[Proposition 7.23]{BP} that there is a non-trivial product in $\widetilde{H}^*(\Z_C;\Bbbk)$, and then so is in $\widetilde{H}^*(\Z_L;\Bbbk)$. This contradicts to the assumption by the ring isomorphism \eqref{Tor}. Hence we get that the 1-skeleton of $L$ is chordal, that is, every minimal cycle in $L$ is of length $\le 3$. In particular we get that the 2-skeleton of $\widehat{L}$ is isomorphic to the 2-skeleton of the flag complex of a chordal graph. Thus since the flag complex of a connected chordal graph is contractible, $|\widehat{L}|$ is simply connected.
\end{proof}

The dual SCMness is preserved by vertex deletions as well as dual shellability.

\begin{proposition}
\label{dl-SCM}
If the Alexander dual of $K$ is SCM over $\Bbbk$, then so is $\dl_K(v)^\vee$ for any $v\in[m]$.
\end{proposition}

\begin{proof}
By Lemma \ref{duality-dictionary}, $\dl_K(v)^\vee=\lk_{K^\vee}(v)$, and by Proposition \ref{SCM-def}, $\lk_{K^\vee}(v)$ is SCM over $\Bbbk$.
\end{proof}

Then in particular, any connected component of a dual SCM complex over $\Bbbk$ is dual SCM over $\Bbbk$. So by Proposition \ref{Golod} and \ref{Golod-1-conn}, we get:

\begin{corollary}
\label{SCM-1-conn}
If the Alexander dual of $K$ is SCM over some ring, then every connected component of $|\widehat{K}|$ is simply connected.
\end{corollary}

We now prove homology fillability of dual SCM complexes over $\mathbb{Z}$.

\begin{proposition}
\label{SCM-homology-fillable}
If the Alexander dual of $K$ is SCM over $\mathbb{Z}$, then $K$ is homology fillable.
\end{proposition}

\begin{proof}
As mentioned above, each connected component of $K$ is dual SCM, so since the homology fillability is a condition for connected components, we may assume $K$ is connected. Fix a prime $p$. By Proposition \ref{homology-spanning}, $K^\vee$ has homology spanning facets $F_1^p,\ldots,F_{r_p}^p$ over $\mathbb{Z}/p$. Put $\Delta^p=K^\vee-\{F_1^p,\ldots,F_{r_p}^p\}$. Since $\partial F_i^p\subset\Delta^p$, $F_i^p$ is a minimal non-face of $\Delta^p$ for all $i$. Then as in the proof of Proposition \ref{shellable-fillable}, we have
$$(\Delta^p)^\vee=K\cup(F_1^p)^\vee\cup\cdots\cup(F_{r_p}^p)^\vee$$
and $(F_i^p)^\vee$ is a minimal non-face of $K$ for all $i$ by Lemma \ref{duality-dictionary}. It follows from Theorem \ref{duality} that
$$\widetilde{H}_i((\Delta^p)^\vee;\mathbb{Z}/p)\cong\widetilde{H}^{m-i-3}(\Delta^p;\mathbb{Z}/p)$$
implying $\widetilde{H}_*((\Delta^p)^\vee;\mathbb{Z}/p)=0$. Then the first condition of the homology fillability is satisfied. The second condition is also satisfied by Corollary \ref{SCM-1-conn}, completing the proof.
\end{proof}

Summarizing, we have established:

\begin{theorem}
\label{SCM-homology-fillable-sub}
If the Alexander dual of $K$ is SCM over $\mathbb{Z}$, then $K_I$ is homology fillable for any $\emptyset\ne I\subset[m]$.
\end{theorem}

\begin{proof}
Combine Proposition \ref{dl-SCM} and \ref{SCM-homology-fillable}.
\end{proof}

\begin{corollary}
[Theorem \ref{main-SCM}]
\label{SCM-trivial}
If the Alexander dual of $K$ is SCM over $\mathbb{Z}$, then the fat wedge filtration of $\RZ_K$ is trivial.
\end{corollary}

\begin{proof}
Combine Corollary \ref{homology-fillable-decomp} and Theorem \ref{SCM-homology-fillable-sub}.
\end{proof}

\begin{corollary}
[Corollary \ref{main-SCM-sphere}]
If the Alexander dual of $K$ is SCM over $\mathbb{Z}$, then $\Z_K(D^n,S^{n-1})$ is homotopy equivalent to a wedge of spheres.
\end{corollary}

\begin{proof}
Combine Corollary \ref{Z-homology-fillable-sphere} and Theorem \ref{SCM-homology-fillable-sub}.
\end{proof}


\section{Extractible complexes}

We have showed that the suspension of a dual shellable complex is homotopy equivalent to a wedge of spheres, and then proved that dual shellability implies the fillability. This homotopy equivalence actually has a further property regarding vertex deletions as follows. Suppose $K$ is dual shellable such that $F_1,\ldots,F_t$ is a shelling of $K^\vee$. If $F_{i_1},\ldots,F_{i_s}$ are all facets  including the vertex $v$ of $K^\vee$ with $i_1<\ldots<i_s$, then as in the proof of Lemma \ref{shellable-lk}, $F_{i_1}-v,\ldots,F_{i_s}-v$ is a shelling of $\lk_{K^\vee}(v)$. Suppose $F_{j_1},\ldots,F_{j_r}$ are spanning facets of $K^\vee$ such that $F_{j_1},\ldots,F_{j_q}$ include the vertex $v$ for $q\le r$, so $F_{j_1}-v,\ldots,F_{j_q}-v$ are, not necessarily all, spanning facets of $\lk_{K^\vee}(v)$. Then the proof of Proposition \ref{shellable-fillable} shows that there are homotopy equivalences
$$|\Sigma K|\simeq|K\cup F_{j_1}^\vee\cup\cdots\cup F_{j_r}^\vee|/|K|=\bigvee_{i=1}^rS^{m-|F_{j_i}|-1}$$
and 
$$|\Sigma\dl_K(v)|\simeq|\dl_K(v)\cup(F_{j_1}-v)^\vee\cup\cdots\cup(F_{j_q}-v)^\vee\cup G_1\cup\cdots\cup G_u|/|\dl_K(v)|=\bigvee_{i=1}^qS^{m-|F_{j_i}|-1}\vee\bigvee_{i=1}^uS^{|G_i|-1}$$
for some minimal non-faces $G_1,\ldots,G_u$ of $\dl_K(v)$, where $F_i^\vee=[m]-F_i$ and $(F_i-v)^\vee=([m]-v)-(F_i-v)=[m]-F_i$ as above. Then the inclusion $|\Sigma\dl_K(v)|\to|\Sigma K|$ restricts to the inclusion $\bigvee_{i=1}^qS^{m-|F_{j_i}|-1}\to\bigvee_{i=1}^rS^{m-|F_{j_i}|-1}$. Thus we can easily deduce the following.

\begin{proposition}
\label{right-inverse}
If the Alexander dual of $K$ is shellable, then the wedge of inclusions 
$$\bigvee_{v\in[m]}|\Sigma\dl_K(v)|\to|\Sigma K|$$
admits a right homotopy inverse.
\end{proposition}

This section proves a homological generalization of the property of Proposition \ref{right-inverse} guarantees a $p$-local desuspension of the BBCG decomposition under some conditions on $\underline{X}$. Before generalizing the property of Proposition \ref{right-inverse}, it is helpful to recall some properties of the Bousfield-Kan (almost) localization. For a space $X$, there is a canonical homotopy equivalence
$$\Sigma X\simeq S\vee Y$$
where $S$ is a bouquet of circles and $Y$ is simply connected. Then we can define the almost $p$-localization of $\Sigma X$ by
$$(\Sigma X)_{(p)}:=S\vee Y_{(p)}$$
which is natural with respect to $X$. Although we ambiguously use the same notation for the usual localization and the almost localization of suspensions, there will be no confusion since we will deal only with simply connected spaces except for suspensions. We freely use the following properties of the (almost) $p$-localization, where the property of the usual localization of wedges of simply connected spaces was already used in Section 5 implicitly. 

\begin{proposition}
[Bousfield and Kan {\cite[Proposition 4.6, Chapter V]{BK}}]
\begin{enumerate}
\item If $X$ and $Y$ are simply connected, 
$$(X\vee Y)_{(p)}\simeq X_{(p)}\vee Y_{(p)}\quad\text{and}\quad(X\wedge Y)_{(p)}\simeq X_{(p)}\wedge Y_{(p)}.$$
\item For any spaces $X,Y$ and a simply connected space $Z$, it holds that 
$$(\Sigma X)_{(p)}\vee(\Sigma Y)_{(p)}\simeq\Sigma(X\vee Y)_{(p)}\quad\text{and}\quad(\Sigma X)_{(p)}\wedge Z_{(p)}\simeq(\Sigma X\wedge Z)_{(p)}.$$
\end{enumerate}
\end{proposition}

We now generalize the property of Proposition \ref{right-inverse}, and define extractible complexes.

\begin{definition}
A simplicial complex $K$ is extractible over $\mathbb{Z}/p$ if
\begin{enumerate}
\item the vertex deletion $\dl_K(v)$ is a simplex for some vertex $v$, or
\item there is a map $\theta\colon|\Sigma K|_{(p)}\to\bigvee_{i=1}^m|\Sigma\dl_K(i)|_{(p)}$ such that the composite
$$|\Sigma K|_{(p)}\xrightarrow{\theta}\bigvee_{i=1}^m|\Sigma\dl_K(i)|_{(p)}\to|\Sigma K|_{(p)}$$
is the identity map in mod $p$ homology, where the second map is the wedge of inclusions.
\end{enumerate}
\end{definition}

By Proposition \ref{right-inverse}, dual shellable complexes are extractible over $\mathbb{Z}/p$ for any prime $p$. So it is natural to ask whether dual SCM complexes over $\mathbb{Z}/p$ are extractible over $\mathbb{Z}/p$ or not. We shall prove the answer is yes. For a chain $x=\sum_{\sigma\in K}a_\sigma\sigma$ ($a_\sigma\in\Bbbk$) of a simplicial complex $K$, we put $x_v=\sum_{\sigma\in K,\,v\not\in\sigma}a_\sigma\sigma$. Note that if a cycle $x$ of $K$ includes a facet of $K$, then $x$ is not a boundary. We consider a relation between cycles of $K$ and of $\lk_K(v)$.

\begin{lemma}
\label{dl-cycle}
Let $x$ be a cycle of $K$ over a ring $\Bbbk$ which involves a facet $F$. For $v\in F$, $\partial(x_v)$ is a cycle of $\lk_K(v)$ over $\Bbbk$ involving $F-v$ which is a facet of $\lk_K(v)$. 
\end{lemma}

\begin{proof}
In the Mayer-Vietoris exact sequence
$$\cdots\to H_{*+1}(K;\Bbbk)\xrightarrow{\delta}H_*(\lk_K(v);\Bbbk)\to H_*(\dl_K(v);\Bbbk)\oplus H_*(\mathrm{st}_K(v);\Bbbk)\to H_*(K;\Bbbk)\to\cdots$$
we have $\delta x=\partial(x_v)$ for the boundary map $\partial$, so the proof is completed by an easy inspection.
\end{proof}

\begin{proposition}
\label{SCM-extractible}
If the Alexander dual of $K$ is SCM over $\mathbb{Z}/p$, $K$ is extractible over $\mathbb{Z}/p$.
 \end{proposition}

\begin{proof}
The argument is quite similar to the dual shellable case. By Proposition \ref{dl-SCM} and the definition of extractibility, we only need to prove the proposition for each connected component of $K$, so we may assume that $K$ is connected. It follows from Proposition \ref{homology-spanning} that there are homology spanning facets $F_1,\ldots,F_r$ of $K^\vee$ over $\mathbb{Z}/p$. Suppose that $F_{i_1},\ldots,F_{i_s}$ involves a vertex $v$. Then it follows from Lemma \ref{dl-cycle} that there are homology spanning facets of $\lk_K(v)$ over $\mathbb{Z}/p$ including $F_{i_1}-v,\ldots,F_{i_s}-v$. Then $K\cup F_1^\vee\cup\cdots\cup F_r^\vee$ and $\dl_K(v)\cup(F_{i_1}-v)^\vee\cup\cdots\cup(F_{i_s}-v)^\vee\cup G_1\cup\cdots\cup G_t$ are acyclic over $\mathbb{Z}/p$ for some minimal non-faces $G_1,\ldots,G_t$ of $\dl_K(v)$, where $(F_i-v)^\vee=([m]-v)-(F_i-v)=[m]-F_i$. Then since these simplicial complexes are of finite type, they are also acyclic over $\mathbb{Z}_{(p)}$, so they are contractible after $p$-localization. Hence as well as the dual shellable case, there are homotopy equivalences
$$|\Sigma K|_{(p)}\simeq\bigvee_{i=1}^rS^{m-|F_i|-1}_{(p)}\quad\text{and}\quad|\Sigma\dl_K(v)|_{(p)}\simeq\bigvee_{k=1}^sS^{m-|F_{i_k}|-1}_{(p)}\vee\bigvee_{i=1}^tS^{|G_i|-1}_{(p)}$$
such that the inclusion $|\Sigma\dl_K(v)|\to|\Sigma K|$ restricts to the inclusion $(\bigvee_{k=1}^sS^{m-|F_{i_k}|-1})_{(p)}\to(\bigvee_{i=1}^rS^{m-|F_i|-1})_{(p)}$. Now the construction of the map $\theta$ is straightforward.
\end{proof}

We now prove a $p$-local desuspension of the BBCG decomposition for extractible complexes over $\mathbb{Z}/p$ by assuming some conditions on $\underline{X}$. 

\begin{theorem}
\label{Z-extractible}
Suppose that each $X_i$ is a connected CW-complex. If $K$ is extractible over $\mathbb{Z}/p$, then there is a homotopy equivalence
$$\Z_K(C\underline{X},\underline{X})_{(p)}\simeq\bigvee_{\emptyset\ne I\subset[m]}(|\Sigma K_I|\wedge\widehat{X}^I)_{(p)}.$$
\end{theorem}

\begin{proof}
First of all, recall from \cite{IK2} that the homotopy equivalence of Corollary \ref{BBCG-natural} is given by the composite of maps
$$\Sigma\Z_K(C\underline{X},\underline{X})\to\Sigma\bigvee^{2^m-1}\Z_K(C\underline{X},\underline{X})\xrightarrow{\rm proj}\Sigma\bigvee_{\emptyset\ne I\subset[m]}\Z_{K_I}(C\underline{X}_I,\underline{X}_I)\xrightarrow{\rm proj}\Sigma\bigvee_{\emptyset\ne I\subset[m]}|\Sigma K_I|\wedge\widehat{X}^I$$
which we denote by $\rho_K$, where the first map is defined by the suspension comultiplication. Since each $X_i$ is a connected CW-complex, $\Z_K(C\underline{X},\underline{X})$ and $\bigvee_{\emptyset\ne I\subset[m]}|\Sigma K_I|\wedge\widehat{X}^I$ are simply connected CW-complexes by Proposition \ref{Z-1-conn}. Then in order to prove the theorem, it is sufficient to construct a map 
$$\epsilon_K\colon\bigvee_{\emptyset\ne I\subset[m]}(|\Sigma K_I|\wedge\widehat{X}^I)_{(p)}\to\Z_K(C\underline{X},\underline{X})_{(p)}$$
which coincides with $\rho_K^{-1}$ in homology. 

We consider the first case in the definition of extractible complexes. By Theorem \ref{main-decomp} and Proposition \ref{dim=m-2}, $\Z_K(C\underline{X},\underline{X})$ is a suspension, so we can define the composite
$$\Z_K(C\underline{X},\underline{X})\to\Sigma\bigvee^{2^m-1}\Z_K(C\underline{X},\underline{X})\xrightarrow{\rm proj}\bigvee_{\emptyset\ne I\subset[m]}\Z_{K_I}(C\underline{X}_I,\underline{X}_I)\xrightarrow{\rm proj}\bigvee_{\emptyset\ne I\subset[m]}|\Sigma K_I|\wedge\widehat{X}^I$$
which we denote by $\bar{\rho}_K$, where the first arrow is defined by the comultiplication of $\Z_K(C\underline{X},\underline{X})$. Obviously $\Sigma\bar{\rho}_K$ is homotopic to $\rho_K$. Then $\bar{\rho}_K$ is an isomorphism in homology, hence a homotopy equivalence by the J.H.C. Whitehead theorem. Thus $(\bar{\rho}_K)_{(p)}^{-1}$ is the desired map $\epsilon_K$.

We next consider the second case in the definition of extractible complexes. Induct on $m$. For $m=1$, both $\bigvee_{\emptyset\ne I\subset[m]}|\Sigma K_I|\wedge\widehat{X}^I$ and $\Z_K(C\underline{X},\underline{X})$ are contractible, so we put $\epsilon_K$ to be the constant map. Suppose that we have constructed the desired map for extractible complexes over $\mathbb{Z}/p$ with vertices less than $m$. Let 
$$\hat{\epsilon}_K\colon\bigvee_{I\subset[m],\,I\ne\emptyset,[m]}(|\Sigma K_I|\wedge\widehat{X}^I)_{(p)}\to\Z_K(C\underline{X},\underline{X})_{(p)}$$
be a wedge of the composite of maps
$$(|\Sigma K_I|\wedge\widehat{X}^I)_{(p)}\xrightarrow{\epsilon_{K_I}}\Z_{K_I}(C\underline{X}_I,\underline{X}_I)_{(p)}\xrightarrow{\rm incl}\Z_K(C\underline{X},\underline{X})_{(p)}$$
for $\emptyset\ne I\subsetneq[m]$, where we have the map $\epsilon_{K_I}$ by the induction hypothesis. Then by the naturality of $\rho_K$ in Corollary \ref{BBCG-natural}, $\hat{\epsilon}_K$ is the restriction of $\rho_K^{-1}$ in homology. Then by the construction of $\rho_K$, we need only to construct a map $\Theta\colon(|\Sigma K|\wedge\widehat{X}^{[m]})_{(p)}\to\Z_K(C\underline{X},\underline{X})_{(p)}$ such that the composite
$$(|\Sigma K|\wedge\widehat{X}^{[m]})_{(p)}\xrightarrow{\Theta}\Z_K(C\underline{X},\underline{X})_{(p)}\xrightarrow{\rm proj}(\Z_K(C\underline{X},\underline{X})/\Z_K^{m-1}(C\underline{X},\underline{X}))_{(p)}=(|\Sigma K|\wedge\widehat{X}^{[m]})_{(p)}$$
is the identity map in homology with coefficient $\mathbb{Z}_{(p)}$. For $v\in[m]$, define a map $\Theta_v\colon(|\Sigma\dl_K(v)|\wedge\widehat{X}^{[m]})_{(p)}\to\Z_K(C\underline{X},\underline{X})_{(p)}$ by the composite
\begin{multline*}
(|\Sigma\dl_K(v)|\wedge\widehat{X}^{[m]})_{(p)}\to
(|\Sigma\dl_K(v)|\wedge\widehat{X}^{[m]-v})_{(p)}\rtimes X_v\xrightarrow{\rm incl}\bigvee_{\emptyset\ne I\subset[m]-v}(|\Sigma\dl_K(v)_I|\wedge\widehat{X}^I)_{(p)}\rtimes X_v\\
\xrightarrow{\epsilon_{\dl_K(v)}\rtimes 1}\Z_{\dl_K(v)}(C\underline{X}_{[m]-v},\underline{X}_{[m]-v})_{(p)}\rtimes X_v\xrightarrow{\rm incl}(\Z_K(C\underline{X},\underline{X})/CX_v)_{(p)}\simeq\Z_K(C\underline{X},\underline{X})_{(p)}
\end{multline*}
where we use a homotopy equivalence $\Sigma A\rtimes B\simeq\Sigma A\vee(\Sigma A\wedge B)$ for the first arrow. Then the naturality of $\rho_K$ in Corollary \ref{BBCG-natural} shows that $\Sigma\Theta_v$ is homotopic to the composite
$$\Sigma(|\Sigma\dl_K(v)|\wedge\widehat{X}^{[m]})_{(p)}\xrightarrow{\rm incl}\Sigma(|\Sigma K|\wedge\widehat{X}^{[m]})_{(p)}\xrightarrow{\rm incl}\Sigma\bigvee_{\emptyset\ne I\subset[m]}(|\Sigma K_I|\wedge\widehat{X}^I)_{(p)}\xrightarrow{\rho_K^{-1}}\Sigma\Z_K(C\underline{X},\underline{X})_{(p)}.$$
Thus the composite
$$f\colon|\Sigma K|_{(p)}\wedge\widehat{X}^{[m]}_{(p)}\xrightarrow{\theta\wedge 1}\bigvee_{i\in[m]}|\Sigma\dl_K(i)|_{(p)}\wedge\widehat{X}^{[m]}_{(p)}\xrightarrow{\bigvee_{i\in[m]}\Theta_i}\Z_K(C\underline{X},\underline{X})_{(p)}\xrightarrow{\rm proj}(|\Sigma K|\wedge\widehat{X}^{[m]})_{(p)}$$
is the identity map in mod $p$ homology, so it is an isomorphism in homology with coefficient $\mathbb{Z}_{(p)}$ since spaces are of finite type. Hence the above composite is a homotopy equivalence, and therefore $(\theta\wedge 1)\circ f^{-1}$ is the desired map, completing the proof. 
\end{proof}

We obtain a $p$-local homotopy decomposition of $\Z_K(C\underline{X},\underline{X})$ when $K$ is dual SCM over $\mathbb{Z}/p$.

\begin{corollary}
If the Alexander dual of $K$ is SCM over $\mathbb{Z}/p$ and each $X_i$ is a connected CW-complex, then there is a homotopy equivalence
$$\Z_K(C\underline{X},\underline{X})_{(p)}\simeq\bigvee_{\emptyset\ne I\subset[m]}(|\Sigma K_I|\wedge\widehat{X}^I)_{(p)}.$$
\end{corollary}

\begin{proof}
Combine Proposition \ref{SCM-extractible} and Theorem \ref{Z-extractible}.
\end{proof}


\section{Neighborly complexes}

In this section, we consider the inductive triviality of the fat wedge filtrations of the real moment-angle complexes, and show that the high neighborliness of $K$ guarantees the triviality of the fat wedge filtration of $\RZ_K$. We first observe a property of the attaching map $\varphi_{K_I}$ when $\varphi_{K_J}$ is null homotopic for all $\emptyset\ne J\subset[m]$ with $|J|<|I|$. Suppose that $\varphi_{K_J}$ is null homotopic for all $\emptyset\ne J\subset[m]$ with $|J|<i$. Then it follows from Theorem \ref{cone-decomp} that there is a homotopy equivalence
$$\RZ_K^{i-1}\simeq\bigvee_{\emptyset\ne J\subset[m],\,|J|<i}|\Sigma K_J|$$
such that the composite $\RZ_K^{i-1}\simeq\bigvee_{\emptyset\ne J\subset[m],\,|J|<i}|\Sigma K_J|\xrightarrow{\rm proj}|\Sigma K_J|$ is homotopic to the composite $\RZ_K^{i-1}\xrightarrow{\rm proj}\RZ_{K_J}\xrightarrow{\rm proj}|\Sigma K_J|$. 

\begin{lemma}
\label{trivial-varphi-part}
Suppose that $\varphi_{K_J}$ is null homotopic for any $\emptyset\ne J\subset[m]$ with $|J|<i$. Then for $I\subset[m]$ with $|I|=i$, the composite
$$|\Sd K_I|\xrightarrow{\varphi_{K_I}}\RZ_K^{i-1}\simeq\bigvee_{\emptyset\ne J\subset[m],\,|J|<i}|\Sigma K_J|\xrightarrow{\rm proj}|\Sigma K_H|$$
is null homotopic for any $\emptyset\ne H\subset[m]$ with $|H|<i$.
\end{lemma}

\begin{proof}
Since the projection $\RZ_K^{i-1}\to\RZ_{K_H}$ factors as $\RZ_K^{i-1}\to\RZ_{\dl_{K_I}(v)}\to\RZ_{K_H}$ for $v\in I-H$, where $\emptyset\ne H\subset[m]$ with $|H|<i$, it is sufficient to show that the composite $|\Sd K_I|\xrightarrow{\varphi_{K_I}}\RZ_{K_I}^{i-1}\xrightarrow{\rm proj}\RZ_{\dl_{K_I}(v)}$ is null homotopic for all $v\in I$. Consider the join $\{v\}*\dl_{K_I}(v)$ for $v\in I$. Then by the definition of $\varphi_{K_I}$, there is a commutative diagram
$$\xymatrix{|\Sd K_I|\ar[rr]^{\varphi_{K_I}}\ar[d]^{\rm incl}&&\RZ_{K_I}^{i-1}\ar[d]^{\rm incl}\\
|\Sd(\{v\}*\dl_{K_I}(v))|\ar[rr]^{\varphi_{\{v\}*\dl_{K_I}(v)}}&&\RZ_{\{v\}*\dl_{K_I}(v)}^{i-1}.}$$
Then since the projection $\RZ_{K_I}^{i-1}\to\RZ_{\dl_{K_I}(v)}$ factors as $\RZ_{K_I}^{i-1}\xrightarrow{\rm incl}\RZ_{\{v\}*\dl_{K_I}(v)}^{i-1}\xrightarrow{\rm proj}\RZ_{\dl_{K_I}(v)}$, the composite $|\Sd K_I|\xrightarrow{\varphi_{K_I}}\RZ_{K_I}^{i-1}\xrightarrow{\rm proj}\RZ_{\dl_{K_I}(v)}$ factors through a contractible space $|\Sd(\{v\}*\dl_{K_I}(v))|$, completing the proof.
\end{proof}

\begin{proposition}
\label{trivial-varphi}
Suppose that $\varphi_{K_J}$ is null homotopic for all $\emptyset\ne J\subset[m]$ with $|J|<i$. Then for $I\subset[m]$ with $|I|=i$, the composite
$$|\Sd K_I|\xrightarrow{\varphi_{K_I}}\RZ_K^{i-1}\simeq\bigvee_{\emptyset\ne J\subset[m],\,|J|<i}|\Sigma K_J|\to\prod_{\emptyset\ne J\subset[m],\,|J|<i}|\Sigma K_J|$$
is null homotopic, where the last arrow is the inclusion.
\end{proposition}

\begin{proof}
The composite of maps in the statement is the product of the composite of maps in Lemma \ref{trivial-varphi-part}, so we obtain the desired result.
\end{proof}

Let $F_i$ be the homotopy fiber of the inclusion $\bigvee_{\emptyset\ne J\subset[m],\,|J|<i}|\Sigma K_J|\to\prod_{\emptyset\ne J\subset[m],\,|J|<i}|\Sigma K_J|$. Then as an immediate corollary of Proposition \ref{trivial-varphi}, we get: 

\begin{corollary}
\label{lift-F}
Suppose that $\varphi_{K_J}$ is null homotopic for all $\emptyset\ne J\subset[m]$ with $|J|<i$. Then for $I\subset[m]$ with $|I|=i$, the attaching map $\varphi_{K_I}$ lifts to $F_i$.
\end{corollary}

Then we check the triviality of the attaching map $\varphi_{K_I}$ for $I\subset[m]$ with $|I|=i$ by looking at its lift to $F_i$ together with Corollary \ref{lift-F} and Proposition \ref{F}. The easiest case that this lift is trivial, is when the connectivity of $F_i$ exceeds the dimension of $K_I$, which we record here.

\begin{proposition}
\label{induct-trivial}
Suppose that $\varphi_{K_J}$ is null homotopic for all $J\subset[m]$ with $|J|<i$. For $I\subset[m]$ with $|I|=i$, if $\dim K_I\le\mathrm{conn}\,F_i$, then the attaching map $\varphi_{K_I}$ is null homotopic.
\end{proposition}

In order to apply Proposition \ref{induct-trivial}, we describe the homotopy type of the homotopy fiber $F_i$. 

\begin{proposition}
\label{F}
The homotopy fiber $F_i$ is homotopy equivalent to 
$$\bigvee_{r\ge 2}\left(\bigvee_{\substack{I_1,\ldots,I_r\subset[m],\\|I_1|<i,\ldots,|I_r|<i}}\bigvee^{r-1}\Sigma(\Omega|\Sigma K_{I_1}|\wedge\cdots\wedge\Omega|\Sigma K_{I_r}|)\right).$$
\end{proposition}

\begin{proof}
This can be proved by the Ganea fibration
$$\Omega X*\Omega Y\to X\vee Y\to X\times Y$$
together with the Hilton-Milnor theorem. We here give an alternative proof using polyhedral products. Let $L$ be the discrete simplicial complex on the vertex set $\{J\subset[m]\,\vert\,|J|<i\}$, and let $\underline{K}$ be a collection of spaces $\{|K_J|\}_{J\in L}$, where we put $|K_\emptyset|$ to be a point. Then as in Example \ref{wedge-example}, we have
$$\Z_L(\Sigma\underline{K},*)=\bigvee_{J\subset[m],\,|J|<i}|\Sigma K_J|$$
and hence $F_i$ is the homotopy fiber of the inclusion $\Z_L(\Sigma\underline{K},*)\to\prod_{J\in L}|\Sigma K_J|$. So by Lemma \ref{fibration}, $F_i$ is homotopy equivalent to the polyhedral product $\Z_L(C\Omega\underline{K},\Omega\underline{K})$. Now $L^\vee$ is a skeleton of a simplex, hence shellable. Thus the proof is done by Theorem \ref{main-decomp} and Corollary \ref{shellable-trivial}.
\end{proof}

\begin{corollary}
\label{conn-F}
The homotopy fiber $F_i$ is $2(\min\{\mathrm{conn}\,K_J\,\vert\,J\subset[m],\,|J|<i\}+1)$-connected.
\end{corollary}
 
 We recall a certain class of simplicial complexes.

\begin{definition}
\label{def-neighborly}
A simplicial complex $K$ is $k$-neighborly if any subset $I\subset[m]$ with $|I|=k+1$ is a simplex of $K$, that is, $K$ includes the $k$-skeleton of $\Delta^{[m]}$. 
\end{definition}

The property of $k$-neighborly complexes that we are going to use is the following.

\begin{lemma}
\label{neighborly-conn}
A simplicial complex is $k$-neighborly if and only if any of its full subcomplex is $(k-1)$-connected.
\end{lemma}

\begin{proof}
Suppose $K$ is $k$-neighborly. Then its $k$-skeleton is the $k$-skeleton of the full simplex $\Delta^{[m]}$ which is $(k-1)$-connected. Any map $S^n\to|K|$ factors through the $n$-skeleton of $|K|$ by the cellular approximation theorem, so if $K$ is $k$-neighborly, then $K$ is $(k-1)$-connected. Since any full sub complex of $K$ is also $k$-neighborly by definition, the proof of the if part is done.

Suppose any full subcomplex of $K$ is $(k-1)$-connected. Let $M$ be a minimal non-face of $K$. Then we have $K_M=\partial\Delta^M$ which is not $(|M|-2)$-connected, implying $|M|>k+1$. Thus $K$ is $k$-neighborly, completing the proof.
\end{proof}

\begin{theorem}
[Theorem \ref{main-neighborly}]
\label{neighborly}
If $K$ is $\lceil\frac{\dim K}{2}\rceil$-neighborly, then the fat wedge filtration of $\RZ_K$ is trivial.
\end{theorem}

\begin{proof}
By Lemma \ref{neighborly-conn} any full subcomplex of $K$ is $(\lceil\frac{\dim K}{2}\rceil-1)$-connected. Let $I\subset[m]$ with $|I|=i$. We prove the triviality of the attaching map $\varphi_{K_I}$ by induction on $i$. For $i=1$, $\varphi_{K_I}$ is obviously trivial. Suppose that $\varphi_{K_J}$ is null homotopic for all $\emptyset\ne J\subset[m]$ with $|J|<i$. By assumption and Corollary \ref{conn-F}, the connectivity of the homotopy fiber $F_i$ is greater than the dimension of $K_I$ since $\dim K\ge\dim K_I$, so by Proposition \ref{induct-trivial}, $\varphi_{K_I}$ is null homotopic, completing the proof.
\end{proof}

\begin{example}
Let $K$ be the 6 vertex triangulation of $\mathbb{R}P^2$ illustrated below. 

\begin{figure}[htbp]
\begin{center}
\setlength\unitlength{1.2mm} 
\begin{picture}(40,40)
\Thicklines
\drawline(20,0)(0,10)(0,30)(20,40)(40,30)(40,10)(20,0)
\drawline(0,10)(40,10)(20,40)(0,10)
\drawline(20,0)(20,10)
\drawline(0,30)(10,25)
\drawline(40,30)(30,25)
\drawline(20,10)(10,25)(30,25)(20,10)
\put(20,0){\circle*{1.5}}
\put(0,10){\circle*{1.5}}
\put(0,30){\circle*{1.5}}
\put(20,40){\circle*{1.5}}
\put(40,30){\circle*{1.5}}
\put(40,10){\circle*{1.5}}
\put(20,10){\circle*{1.5}}
\put(10,25){\circle*{1.5}}
\put(30,25){\circle*{1.5}}
\put(-3,10){$3$}
\put(-3,30){$2$}
\put(17,40){$1$}
\put(41.3,10){$2$}
\put(41.3,30){$3$}
\put(17,-2){$1$}
\put(-3,10){$3$}
\put(8.5,26.6){$4$}
\put(19.1,12){$5$}
\put(29.5,26.6){$6$}
\end{picture}
\end{center}
\end{figure}

\noindent It is shown in \cite{GPTW} that the BBCG decomposition of the moment-angle complex $\Z_K$ desuspends. However their argument is quite ad-hoc and depends heavily on the pair $(D^2,S^1)$, so it is not applicable to $\Z_K(C\underline{X},\underline{X})$ in general. Now $K$ is 1-neighborly and $\dim K=2$, so we can apply Theorem \ref{main-decomp} and \ref{main-neighborly} to obtain a desuspension of the BBCG decomposition such that
$$\Z_K(C\underline{X},\underline{X})\simeq\left(\bigvee_{I\in S}\Sigma^2\widehat{X}^I\right)\vee\left(\bigvee_{I\subset[6],\,|I|=4,5}\Sigma^2\widehat{X}^I\right)\vee(\Sigma\mathbb{R}P^2\wedge\widehat{X}^{[6]})$$
where $S=\{\{3,5,6\},\{3,4,6\},\{2,4,6\},\{2,4,5\},\{2,3,5\},\{1,5,6\},\{1,4,5\},\{1,3,4\},\{1,2,6\},\{1,2,3\}\}$.
\end{example}

We give a generalization of Theorem \ref{neighborly} by replacing the dimension of $K$ with the homology dimension of $K$.

\begin{definition}
The homology dimension of a space $X$, denoted by $\mathrm{hodim}\,X$, is less than or equal to $n$ if and only if $\widetilde{H}_*(X;A)=0$ for $*>n$ and any finitely generated abelian group $A$. 
\end{definition}

We prepare two technical lemmas.

\begin{lemma}
\label{perfect-free}
If $G$ is a perfect group and $F$ is a free group, then any homomorphism $G\to F$ is trivial. 
\end{lemma}

\begin{proof}
For a homomorphism $f\colon G\to F$ the image $f(G)$ is a perfect subgroup of $F$. By the Nielsen-Schreier theorem, $f(G)$ is also a free group, then $f(G)$ must be trivial.
\end{proof}

\begin{lemma}
\label{trivial-[X,Y]}
Let $X$ be a finite CW-complex and $Y$ be an $n$-connected space of finite type. If $\mathrm{hodim}\,X\le n$ and additionally $\pi_1(Y)$ is free for $n=0$, then any map $X\to Y$ is null homotopic.
\end{lemma}

\begin{proof}
Consider the Postnikov tower of $Y$:
$$\cdots\to Y_k\to Y_{k-1}\to\cdots\to Y_2\to Y_1=K(\pi_1(Y),1)$$
Since $X$ is a finite CW-complex, the triviality of the homotopy set $[X,Y]$ is implied by the triviality of $[X,Y_k]$ for all $k$. It follows from Lemma \ref{perfect-free} that $[X,K(\pi_1(Y),1)]=*$ for $n=0$, and $[X,K(\pi_1(Y),1)]$ is obviously trivial for $n>0$. So the homotopy exact sequence associated with the homotopy fibration $K(\pi_k(Y),k)\to Y_k\to Y_{k-1}$ shows that $[X,Y_k]=*$ for all $k$. 
\end{proof}

Put 
$$d(K):=\max\{\mathrm{hodim}\,K_I\,\vert\,\emptyset\ne I\subset[m]\}.$$
Obviously we have $d(K)\le\dim K$. Quite similarly to Theorem \ref{neighborly} together with Lemma \ref{trivial-[X,Y]}, we can prove the following.

\begin{theorem}
\label{RZ-neighborly}
If $K$ is $\lceil\frac{d(K)}{2}\rceil$-neighborly, then the fat wedge filtration of $\RZ_K$ is trivial.
\end{theorem}

\begin{example}
\label{Berglund-example}
In \cite{B} Berglund considered a simplicial complex $K$ on the vertex set $[10]$ whose minimal non-faces are
$$\{1,2,6,7\},\,\{2,3,7,8\},\,\{3,4,8,9\},\,\{4,5,9,10\},\,\{1,5,6,10\},\,\{6,7,8,9,10\}.$$
It was proved that $K^\vee$ is not SCM over $\mathbb{Z}$ but $K$ is Golod, so we cannot apply Corollary \ref{SCM-trivial} to decompose the polyhedral product $\Z_K(C\underline{X},\underline{X})$. Note that $K$ is 2-neighborly but is not 3-neighborly, and $\dim K=6$. Then we cannot apply Theorem \ref{neighborly} to this case either. We shall show $d(K)\le 4$, and apply Theorem \ref{RZ-neighborly}, which implies that our generalization from $\dim K$ to $d(K)$ is meaningful. 

Let $I$ be a non-empty subset of $[m]$. 


\noindent(1) For $|I|\le 7$, $|K_I|$ is homotopy equivalent to a CW-complex of dimension $\le 4$ since $K_I$ is not the boundary of the 6-simplex. Then $\mathrm{hodim}\,K_I\le 4$. 

\noindent(2) For $|I|=8$, it is a routine work to check that $(K_I)^\vee$ is contractible or homotopy equivalent to $S^1$ since $(K_I)^\vee$ has at most three facets. Then $K_I$ is contractible or homotopy equivalent to $S^4$ by Theorem \ref{duality} and the fact that $K_I$ is simply connected.

\noindent(3) For $|I|\ge 9$, $K_I$ is contractible. The proof is divided into two cases. If $I=[10]-\{i\}$ for $i=6,\cdots,10$, then $K_I$ is a cone which is contractible. For example, $K_{[9]}$ is a cone with apex 5. For the other case, we only consider the whole complex $K$ since other cases are similar. Consider the cofibration
$$|\lk_K(10)|\to|K_{[9]}|\to|K|.$$
Since $|K_{[9]}|$ is contractible, $|K|\simeq\Sigma|\lk_K(10)|$. Similarly we consider the cofibration
$$|\lk_K(\{9, 10\})|\to|(\lk_K(10))_{[8]}|\to|\lk_K(10)|,$$
where $(\lk_K(10))_{[8]}$ is a simplicial complex on the vertex set $[8]$ with the minimal non-faces
$$\{1,2,6,7\},\;\{2,3,7,8\},\;\{1,5,6\}.$$
Then $(\lk_K(10))_{[8]}$ is a cone with apex 4 which is contractible. Then we get $|K|\simeq\Sigma|\lk_K(10)|\simeq\Sigma^2|\lk_K(\{9,10\})|$ as above. Furthermore, we can see that $|K|\simeq\Sigma^4|\lk_K(\{7, 8, 9, 10\})|$ in the same way, where $\lk_K(\{7, 8, 9, 10\})$ is a simplicial complex on the vertex set $[6]$ with the minimal non-faces
$$\{2,3\},\;\{3,4\},\;\{4,5\},\;\{6\}.$$
Since it is a cone with apex 1, $|\lk_K (\{7, 8, 9, 10\})|$ is contractible, and therefore $|K|$ is also contractible. 

Summarizing, we conclude $d(K)=4$.
\end{example}


\section{Further problems}

In this section, we list possible future problems on the homotopy type of the polyhedral product $\Z_K(C\underline{X},\underline{X})$ mainly related with the fat wedge filtrations. 

\subsection{Converse of Theorem \ref{main-decomp}}

The biggest problem concerning Theorem \ref{main-decomp} is:

\begin{problem}
\label{converse}
Is the converse of Theorem \ref{main-decomp} true?
\end{problem}

In Section 5 we have proved Theorem \ref{main-decomp-Z} which is a partial converse of Theorem \ref{main-decomp}. The key in its proof is the simply connectedness of $\Z_K$ for any $K$. So one easily sees that if $\RZ_{K_I}$ is simply connected for any $\emptyset\ne I\subset[m]$, the same proof works for $\RZ_K$, and then by Lemma \ref{neighborly-conn} we get:

\begin{proposition}
\label{1-neighborly}
If $K$ is 1-neighborly, then the converse of Theorem \ref{main-decomp} holds.
\end{proposition}
 
We here propose an approach to Problem \ref{converse} by induction on $m$. The case $m=1$ is trivial, and we suppose the converse of Theorem \ref{main-decomp} holds for simplicial complexes with vertices less than $m$. Then we have $\varphi_{K_I}\simeq*$ for any $\emptyset\ne I\subsetneq[m]$ since $\RZ_{K_I}=\Z_K(C\underline{X},\underline{X})$ for $X_i=S^0$ with $i\in I$ and $X_i=*$ with $i\not\in I$, implying $\RZ_K$ and $\RZ_K^{m-1}$ are suspensions. In particular it is sufficient to consider the case that $K$ is connected similarly to the proof of Theorem \ref{homology-fillable-phi}. As well as the proof of Proposition \ref{1-neighborly} we can easily see that there is a map $\RZ_K\to\RZ_K^{m-1}$ such that the composite $\RZ_K^{m-1}\xrightarrow{\rm incl}\RZ_K\to\RZ_K^{m-1}$ an isomorphism in homology.  Moreover we have:

\begin{lemma}
\label{pi_1}
Suppose $\RZ_K^{m-1}$ and $\RZ_K$ are suspensions. If $K$ is connected, the inclusion $\RZ_K^{m-1}\to\RZ_K$ is an isomorphism in fundamental group.
\end{lemma}

\begin{proof}
It follows from Theorem \ref{cone-decomp} that $\RZ_K=\RZ_K^{m-1}\cup_{\varphi_K}C|\Sd K|$, so by the van Kampen theorem, $\pi_1(\RZ_K)\cong\pi_1(\RZ_K^{m-1})/N$ and the inclusion $\RZ_K^{m-1}\to\RZ_K$ in fundamental group is identified with the quotient map 
$$\pi_1(\RZ_K^{m-1})\to\pi_1(\RZ_K^{m-1})/N$$
where $N$ is the smallest normal subgroup including $\mathrm{Im}\,\{(\varphi_K)_*\colon\pi_1(|\Sd K|)\to\pi_1(\RZ_K^{m-1})\}$. By Proposition \ref{suspension-varphi}, $\Sigma\RZ_K\simeq\Sigma\RZ_K^{m-1}\vee\Sigma^2|\Sd K|$, implying $H_1(\RZ_K^{m-1})\cong H_1(\RZ_K)$ since $K$ is connected. Then since $\RZ_K^{m-1}$ and $\RZ_K$ are suspensions, their fundamental groups are free groups of the same rank. Therefore the inclusion $\RZ_K^{m-1}\to\RZ_K$ is an isomorphism by the above observation since free groups are Hopfian.
\end{proof}

Then we can easily modify the above map $\RZ_K\to\RZ_K^{m-1}$ such that the composite $\RZ_K^{m-1}\xrightarrow{\rm incl}\RZ_K\to\RZ_K^{m-1}$ is an isomorphism in $\pi_1$ and homology. Thus one way to resolve Problem \ref{converse} is to prove the following conjecture, where it is well known that a map inducing an isomorphism in $\pi_1$ and homology needs not be a homotopy equivalence.

\begin{conjecture}
Let $A,X,Y$ be finite CW-complexes. Suppose a map $f\colon\Sigma X\to\Sigma Y$ is an isomorphism in $\pi_1$ and homology. If a map $g\colon A\to\Sigma X$ satisfies $f\circ g\simeq*$, then $g\simeq*$.
\end{conjecture}

\subsection{Homotopy Golodness}

In Section 6, we have proved that there is an implication:
\begin{equation}
\label{implication-1}
\text{triviality of the fat wedge filtration of }\RZ_K\quad\Longrightarrow\quad\text{homotopy Golodness of }K
\end{equation}
The proof of this implication only deals with the top filter of the associated fat wedge filtration of the real moment-angle complexes, so the implication seems to be strict. Then it is worth studying the gap of this implication to get a further interpretation of the fat wedge filtrations, so we ask the following, where, of course, we can consider the stable homotopy Golodness analogue of this problem for $\Z_K$.

\begin{problem}
Find a simplicial complex for which the implication \eqref{implication-1} is strict or equality.
\end{problem}

The Golodness is defined by the triviality of certain maps in homology, and the (resp. stable) homotopy Golodness is defined by replacing the triviality in homology with the triviality up to (resp. stable) homotopy. Then as in Proposition \ref{Golod-implication} we have implications:
\begin{equation}
\label{implication-2}
\text{homotopy Golodness}\quad\Longrightarrow\quad\text{stable homotopy Golodness}\quad\Longrightarrow\quad\text{Golodness}
\end{equation}
We next ask the following question which seems quite combinatorial. 

\begin{problem}
\label{implication-problem}
Find simplicial complexes for which the implications \eqref{implication-2} are strict or equality.
\end{problem}

Here we give examples of a class of simplicial complexes for which all implications in \eqref{implication-1} and \eqref{implication-2} are equalities.

\begin{theorem}
Suppose $K$ is a flag complex. Then the fat wedge filtration of $\RZ_K$ is trivial if and only if $K$ is Golod over $\mathbb{Z}$.
\end{theorem}

\begin{proof}
If the fat wedge filtration of $\RZ_K$ is trivial, then $K$ is Golod over any ring by Theorem \ref{main-decomp} and Proposition \ref{BBCG-Golod}. Then we show the converse holds. As mentioned in the proof of Proposition \ref{Golod-1-conn}, if the underlying graph of $K$ is not chordal, $K$ is not Golod. Thus the proof is completed by Proposition \ref{flag-varphi}.
\end{proof}

\begin{theorem}
\label{graph}
If $\dim K=1$, then the following conditions are equivalent:
\begin{enumerate}
\item The fat wedge filtration of $\RZ_K$ is trivial;
\item $K$ is Golod over some ring;
\item $K$ is chordal.
\end{enumerate}
\end{theorem}

\begin{proof}
By Proposition \ref{Golod-implication} and Theorem \ref{main-Golod}, (1) implies (2), and the proof of Proposition \ref{Golod-1-conn} shows that (2) implies (3). Then we show (3) implies (1). Induct on $m$. The case $m=1$ is trivial, and we assume $\varphi_{K_I}\simeq*$ for any $\emptyset\ne I\subsetneq[m]$. The case $K$ is disconnected is done since $\varphi_K$ factors through $\varphi_{K_1}\sqcup\cdots\sqcup\varphi_{K_\ell}$ and $\varphi_{K_i}\simeq*$ for any $i$ by the induction hypothesis, where $K_1,\ldots,K_\ell$ are the connected components of $K$. Then we assume $K$ is connected. Since $|\Sd\,K|$ is homotopy equivalent to a wedge of circles, it is sufficient to show that $\varphi_K$ is trivial in $\pi_1$. The proof of Theorem \ref{fillable-1} shows that $\varphi_K$ factors through the 2-skeleton of the flag complex of $K$. Since the flag complex of a chordal graph is contractible, the proof is completed.
\end{proof}

\begin{remark}
Recently, L. Katth\"an \cite{K} gave an alternative algebraic proof for the equivalence of (2) and (3) of Theorem \ref{graph}.
\end{remark}

Continuing Theorem \ref{graph}, it might be interesting to consider Problem \ref{implication-problem} when $\dim K=2$. 

\begin{problem}
When $\dim K=2$, find combinatorial conditions of $K$ such that
\begin{enumerate}
\item $K$ is Golod over a ring $\Bbbk$ or 
\item the fat wedge filtration of $\RZ_K$ is trivial.
 \end{enumerate}
\end{problem}

Notice that by Theorem \ref{RZ-neighborly} 1-neighborliness is a sufficient condition for (2), hence (1), for 2-dimensional $K$. But unlikely to Theorem \ref{graph} even if the 1-skeleton of $K$ is chordal, $K$ is not necessarily Golod. For example the 1-skeleton of $\partial\Delta^{[2]}*\partial\Delta^{[3]}$ is chordal but $\Z_K=S^3\times S^5$.

\subsection{Strong gcd-condition}

We first pose a general problem.

\begin{problem}
\label{general-problem}
Find a class of simplicial complexes for which the fat wedge filtrations of the real moment-angle complexes are trivial.
\end{problem}

One of our choice for the above problem in this paper is dual SCM complexes, where the choice is motivated by the Golodness. As mentioned above, SCM complexes can be thought of as a generalization of shellable complexes, and there is another generalization of dual shellability, called the strong gcd-condition introduced by J\"ollenbeck \cite{Jo}. We here recall the definition of the strong gcd-condition.

\begin{definition}
A simplicial complex $K$ satisfies the strong gcd-condition if minimal non-faces of $K$ admit an ordering $M_1,\ldots,M_r$, called a strong gcd-order, such that whenever $1\le i<j\le r$ and $M_i\cap M_j=\emptyset$, $M_k\subset M_i\cup M_j$ for some $k$ with $i<k\ne j$. 
\end{definition}

The following proposition guarantees that the strong gcd-condition is a generalization of dual shellability.

\begin{proposition}
[Berglund \cite{B}]
The Alexander dual of $K$ satisfies the strong gcd-condition whenever it is shellable.
\end{proposition}

In \cite{Jo} it is claimed that if a simplicial complex satisfies the strong gcd-condition, then it is Golod. But unfortunately, a counter example to this claim was recently found by De Stefani \cite{D}. We here show another simple counter example which is due to T. Yano. Let $L$ be a simplicial complex on the vertex set $[6]$ whose facets are
$$\{1,2,3\},\;\{1,3,4\},\;\{1,4,5\},\;\{1,2,5\},\;\{2,3,6\},\;\{3,4,6\},\;\{4,5,6\},\;\{2,5,6\},\;\{2,4\},\;\{3,5\}.$$
Then $L$ is the octahedron with two additional edges, and it is immediate to check that $L$ satisfies the strong gcd-condition. On the other hand, for $I=\{1,6\}$ and $J=\{2,3,4,5\}$, the inclusion $L_I*L_J\to L_{I\cup J}=L$ is non-trivial in homology with any coefficient. So by the description of the products in $\mathrm{Tor}_{\Bbbk[v_1,\ldots,v_m]}(\Bbbk[K],\Bbbk)$ in Section 6, $L$ is not Golod over any ring.

However, simplicial complexes satisfying the strong gcd-condition are still interesting to investigate, so we pose:

\begin{problem}
\label{gcd}
\begin{enumerate}
\item Find a class of simplicial complexes which satisfy the strong gcd-condition and are Golod.
\item Show that the fat wedge filtration of $\RZ_K$ is trivial when $K$ is in the above class.
\end{enumerate}
\end{problem}

We here give some examples of simplicial complexes satisfying the strong gcd-condition for which the fat wedge filtrations of the real moment-angle complexes are trivial. It is useful to recall from \cite{B} the weak shellability which is the Alexander dual of the strong gcd-condition.

\begin{definition}
A simplicial complex $K$ is called weakly shellable if there is an ordering $F_1,\ldots,F_r$ of the facets of $K$, called a weak shelling, such that if $F_i\cup F_j=[m]$ for $i<j$, then there is $F_i\cap F_j\subset F_k$ for some $k$ with $i\ne k<j$.
\end{definition}

\begin{proposition}
[Berglund \cite{B}]
\label{weak-shellable-gcd}
An ordering $M_1,\ldots,M_r$ of subsets of $[m]$ is a strong gcd-order of $K$ if and only if the ordering $M_r^\vee,\ldots,M_1^\vee$ is a weak shelling of $K^\vee$.
\end{proposition}

If $2\dim K^\vee+2<m$, then $K^\vee$ is weakly shellable by any ordering of facets, hence $K$ satisfies the strong gcd-condition. 

\begin{proposition}
\label{low-dim}
If $2\dim K^\vee+2<m$, then the fat wedge filtration of $\RZ_K$ is trivial.
\end{proposition}

\begin{proof}
If $\dim K\ge m-2$, the fat wedge filtration of $\RZ_K$ is trivial by Proposition \ref{dim=m-2}, so we assume $\dim K\le m-3$. Since all simplices of $K^\vee$ are of dimension at most $d=\dim K^\vee$, all $(m-d-3)$-dimensional simplices of $\Delta^{[m]}$ belong to $K$, hence $K$ is $(m-d-3)$-neighborly. Since $2d+2<m$, we have $\lceil\frac{m-3}{2}\rceil\le m-d-3$. Thus by Theorem \ref{neighborly} the fat wedge filtration of $\RZ_K$ is trivial. 
\end{proof}

\begin{remark}
One can easily see that if $K^\vee$ is connected, the condition $2\dim K^\vee+2<m$ in Proposition \ref{low-dim} can be improved by one such that $2\dim K^\vee+1<m$.
\end{remark}

\begin{corollary}
If $K^\vee=\Sd L$ for a simplicial complex $L$ with $\dim L\ge 2$, then the fat wedge filtration of $\RZ_K$ is trivial. 
\end{corollary}

\begin{proof}
If $\dim L=d$, then $\Sd L$ has at least $2^{d+1}-1$ vertices, and for $d\ge 2$, we have $2d+2<2^{d+1}-1\le m$. Thus since $\dim\Sd L=\dim L=d$, the proof is completed by Proposition \ref{low-dim}.
\end{proof}


\begin{thebibliography}{BW}
\bibitem[BBCG]{BBCG}A. Bahri, M. Bendersky, F.R. Cohen, and S. Gitler, {\it The polyhedral product functor: a method of decomposition for moment-angle complexes, arrangements and related spaces}, Advances in Math. {\bf 225} (2010), 1634-1668.
\bibitem[BBP]{BBP}I.V. Baskakov, V.M. Buchstaber, and T.E. Panov, {\it Cellular cochain algebras and torus actions}, Russian Math. Surveys {\bf 59} (2004), no. 3, 562-563.
\bibitem[B]{B}A. Berglund, {\it Shellability and the strong gcd-condition}, Electronic J. Comb. [electronic only], {\bf 16} (2009).
\bibitem[BJ]{BJ}A. Berglund and M. J\"ollenbeck, {\it On the Golod property of Stanley-Reisner rings}, J. Algebra {\bf 315} (2007), 246-273.
\bibitem[BT]{BT}A. Bj\"orner and M. Tancer, {Note: Combinatorial Alexander duality- a short and elementary proof}, Discrete \& Computational Geometry {\bf 42} (2009), 586-593.
\bibitem[BW]{BW}A. Bj\"orner and M.I. Wachs, {\it Shellable nonpure complexes and posets. I}, Trans. AMS {\bf 348} (1996), 1299-1327; {\it Shellable nonpure complexes and posets. II}, Trans. AMS {\bf 349} (1997), 3945-3975. 
\bibitem[BWW]{BWW}A. Bj\"orner, M. Wachs, and V. Welker, {\it On Sequentially Cohen-Macaulay complexes and posets}, Israel J. Math. {\bf 169} (2009), 295-316.
\bibitem[BK]{BK}A.K. Bousfield and D.M. Kan, {\it Homotopy Limits, Completions and Localizations}, Lecture Notes in Mathematics {\bf 304}, Springer-Verlag, Berlin, 1972.
\bibitem[BP]{BP}V.M. Buchstaber and T.E. Panov, {\it Torus actions and their applications in topology and combinatorics}, University Lecture Series {\bf 24}, American Mathematical Society, Providence, RI, 2002. 
\bibitem[DJ]{DJ}M.W. Davis and T. Januszkiewicz, {\it Convex polytopes, Coxeter orbifolds and torus actions}, Duke Math. J. {\bf 62} (1991), 417-451.
\bibitem[DO]{DO}M.W. Davis and B. Okun, {\it Cohomology computations for Artin groups, Bestvina-Brady groups, and graph products}, Groups Geom. Dynam. {\bf 6} (2012), no. 3, 485-531.
\bibitem[D]{D}A. De Stefani, {\it Products of ideals may not be Golod},  J. Pure Appl. Algebra {\bf 220} (2016), no. 6, 2289-2306.
\bibitem[DS]{DS}G. Denham and A.I. Suciu, {\it Moment-angle complexes, monomial ideals, and Massey products}, Pure Appl. Math. Q. {\bf 3} (2007), no. 1, Special Issue: In honor of Robert D. MacPherson., 25-60.
\bibitem[Fa]{Fa}E.D. Farjoun, {\it Cellular spaces, null spaces and homotopy localization}, Lecture Notes in Mathematics {\bf 1622}, Springer-Verlag, Berlin, 1996.
\bibitem[Fo]{Fo}R.H. Fox, {\it On the Lusternik-Schnirelmann category}, Ann. of Math. (2) {\bf 42} (1941), 333-370.
\bibitem[G]{G}E.S. Golod, {\it On the homologies of certain local rings}, Soviet Math. Dokl. {\bf 3} (1962), 745-748.
\bibitem[GPTW]{GPTW}J. Grbi\'c, T. Panov, S. Theriault, and J. Wu, {\it Homotopy types of moment-angle complexes for flag complexes},  Trans. Amer. Math. Soc. {\bf 368} (2016), no. 9, 6663-6682.
\bibitem[GT1]{GT1}J. Grbi\'c and S. Theriault, {\it The homotopy type of the complement of a coordinate subspace arrangement}, Topology {\bf 46} (2007), 357-396.
\bibitem[GT2]{GT2}J. Grbi\'c and S. Theriault, {\it The homotopy type of the polyhedral product for shifted complexes}, Advances in Math. {\bf 245} (2013), 690-715.
\bibitem[GW]{GW}V. Gruji\'c and V. Welker, {\it Discrete Morse theory for moment-angle complexes of pairs $(D^n,S^{n-1})$}, Monatsh. Math. {\bf 176} (2015), no. 2, 255-273.
\bibitem[H]{H}M. Hachimori, webpage: \url{http://infoshako.sk.tsukuba.ac.jp/~hachi/math/library/index_eng.html}.
\bibitem[HRW]{HRW}J. Herzog, V. Reiner, and V. Welker, {\it Componentwise linear ideals and Golod rings}, Mich. Math. J. {\bf 46} (1999), 211-223.
\bibitem[HMR]{HMR}P. Hilton, G. Mislin, and J. Roitberg, {\it On co-H-spaces}, Comment. Math. Helvetici {\bf 53} (1978), 1-14.
\bibitem[IK1]{IK1}K. Iriye and D. Kishimoto, {\it Decompositions of polyhedral products for shifted complexes}, Advances in Math. {\bf 245} (2013), 716-736.
\bibitem[IK2]{IK2}K. Iriye and D. Kishimoto, {\it Decompositions of suspensions of spaces involving polyhedral products}, Algebr. Geom. Topol. {\bf 16} (2016), no. 2, 825-841.
\bibitem[Ja]{Ja}I.M. James, {\it On H-spaces and their homotopy groups}, Quart. J. Math. Oxford Ser. (2) {\bf 11} (1960), 161-179.
\bibitem[Jo]{Jo}M. J\"ollenbeck, {\it On the multigraded Hilbert- and Poincar\'e series of monomial rings}, J. Pure Appl. Algebra {\bf 207}, No. 2, (2006), 261-298.
\bibitem[K]{K}L. Katth\"an, {\it The Golod property for Stanley-Reisner rings in varying characteristic},  J. Pure Appl. Algebra {\bf 220} (2016), no. 6, 2265-2276.
\bibitem[P]{P}T. Porter, {\it Higher-order Whitehead products}, Topology {\bf 3} (1965), 123-135.
\bibitem[S]{S}R.P. Stanley, {\it Combinatorics and commutative algebra}, Second edition. Progress in Mathematics {\bf 41}, Birkh\"auser Boston, Inc., Boston, MA, 1996.
\bibitem[ZZ]{ZZ}G.M. Ziegler and R.T. \u{Z}ivaljevi\'c, {\it Homotopy types of subspace arrangements via diagrams of spaces}, Math. Ann. {\bf 295}, (1993), 527-548.
\end{thebibliography}
\end{document}